\documentclass[12pt]{amsart}
\usepackage{graphicx}
\usepackage{amsmath}
\usepackage{amssymb,amscd,latexsym}
\usepackage{bm}
\usepackage[all]{xy}
\newcommand{\llbracket}{[\![}
\newcommand{\rrbracket}{]\!]}
\newcommand{\te}{\textstyle}

\newcommand{\dl}{\mathrm{d}\!\mathrm{l}}
\newcommand{\F}{{\mathbb{F}}}

\newcommand{\N}{{\mathbb{N}}}

\newcommand{\Q}{{\mathbb{Q}}}
\newcommand{\oQ}{\overline{\Q}}

\newcommand{\R}{{\mathbb{R}}}
\newcommand{\Z}{{\mathbb{Z}}}

\newcommand{\halpha}{\hat{\alpha}}

\newcommand{\bfa}{\underline{a}}
\newcommand{\bfb}{\underline{b}}
\newcommand{\bfc}{\underline{c}}

\newcommand{\bfn}{\underline{n}}
\newcommand{\bfr}{\underline{r}}
\newcommand{\bfx}{\underline{x}}

\newcommand{\uomega}{\underline{\omega}}

\newcommand{\alg}{\mathrm{alg}}

\newcommand{\id}{\mathrm{id}}

\newcommand{\re}{\mathrm{ref}}
\renewcommand{\mod}{\;\mathrm{mod}\;}
\newcommand{\proj}{\mathrm{proj}}

\newcommand{\spec}{\mathrm{spec}\,}
\newcommand{\End}{\mathrm{End}\,}
\newcommand{\Hom}{\mathrm{Hom}}

\newcommand{\Ker}{\mathrm{Ker}\,}

\newcommand{\tg}{\tilde{g}}
\newcommand{\tR}{\tilde{R}}

\newcommand{\talpha}{\tilde{\alpha}}

\newcommand{\Fh}{{\mathcal F}}
\newcommand{\Gh}{{\mathcal G}}

\newcommand{\Vh}{{\mathcal V}}

\newcommand{\eo}{\mathfrak{o}}

\newcommand{\oalpha}{\overline{\alpha}}

\newcommand{\oI}{\overline{I}}

\newcommand{\oR}{\overline{R}}

\newcommand{\oF}{\overline{\mathbb{F}}}

\newcommand{\hA}{\widehat{A}}

\newcommand{\oa}{\overline{a}}

\newcommand{\tc}{\tilde{c}}

\newcommand{\hU}{\hat{U}}

\newcommand{\ophi}{\overline{\phi}}

\newcommand{\ohne}{\setminus}
\newcommand{\silo}{\xrightarrow{\sim}}
\newcommand{\tei}{\, | \,}
\newcommand{\bfdot}{\hullet}
\newcommand{\hullet}{\raisebox{0.05cm}{$\,\scriptscriptstyle \bullet\,$}}
\newcommand{\verk}{\mbox{\scriptsize $\,\circ\,$}}

\newtheorem{theorem}{Theorem}[section]
\newtheorem{lemma}[theorem]{Lemma}
\newtheorem{proposition}[theorem]{Proposition}
\newtheorem{definition}[theorem]{Definition}
\newtheorem{defprop}[theorem]{Definition/Proposition}
\newtheorem{cor}[theorem]{Corollary}

\newtheorem{remark}[theorem]{Remark}

\newenvironment{rem}{\noindent {\bf Remark}}{}

\newenvironment{examp}{\noindent {\bf Example}}{}
\newcommand{\beweisende}{\hspace*{\fill} $\Box$}
\begin{document}
\title{ Witt vector rings and the relative de~Rham Witt complex }
\author{Joachim Cuntz and Christopher Deninger}
\thanks{Research supported by DFG through CRC 878, by ERC through AdG 267079 and by the MPIM Bonn}
\maketitle
\begin{center}
with an appendix by {\small \MakeUppercase{Umberto Zannier}}
\end{center}

\begin{abstract}
In this paper we develop a novel approach to Witt vector rings and to the (relative) de Rham Witt complex. We do this in the generality of arbitrary commutative algebras and arbitrary truncation sets. In our construction of Witt vector rings the ring structure is obvious and there is no need for universal polynomials. Moreover a natural generalization of the construction easily leads to the relative de Rham Witt complex. Our approach is based on the use of free or at least torsion free presentations of a given commutative ring $R$ and it is an important fact that the resulting objects are independent of all choices. The approach via presentations also sheds new light on our previous description of the ring of $p$-typical Witt vectors of a perfect $\mathbb{F}_p$-algebra as a completion of a semigroup algebra. We develop this description in different directions. For example, we show that the semigroup algebra can be replaced by any free presentation of $R$ equipped with a linear lift of the Frobenius automorphism. Using the result in the appendix by Umberto Zannier we also extend the description of the Witt vector ring as a completion to all $\overline{\mathbb{F}}_p$-algebras with injective Frobenius map. 
\end{abstract}

\section*{Introduction}

Motivated by the problem of constructing field extensions of degree $p^n$, Witt introduced the ring of Witt vectors of a given ring. In his ingenious construction the ring operations are defined using infinite sequences of universal polynomials. Using Witt vector rings one may for example pass from perfect fields $k$ of positive characteristic to unramified complete discrete valuation rings $R$ with residue field $k$ and quotient field of characteristic zero. By now, Witt vector rings  and their more general variants are a classical tool in many branches of mathematics ranging from algebra and algebraic number theory over arithmetic geometry to homotopy theory. However the construction of the ring of Witt vectors and, even more so, the construction of the associated de~Rham-Witt complex is still nowadays considered to be not so easy.

In this paper we develop a novel approach to Witt vector rings and to the (relative) de~Rham Witt complex. We do this in the generality of arbitrary commutative algebras and arbitrary truncation sets. In particular the big and the $p$-isotypical theories are covered. In our construction of Witt vector rings the ring structure is obvious and there is no need for universal polynomials. Moreover a natural generalization of the construction easily leads to the relative de~Rham Witt complex. \\
Our approach is based on the use of free or at least torsion free presentations of a given commutative ring $R$ and it is an important fact that the resulting objects are independent of all choices. The approach via presentations also sheds new light on the description in \cite{CD} of the ring of $p$-typical Witt vectors of a perfect $\F_p$-algebra as a completion of a semigroup algebra. We develop this description in different directions. For example, we show that the semigroup algebra can be replaced by any free presentation of $R$ equipped with a linear lift of the Frobenius automorphism. Using the result in the appendix by Umberto Zannier we also extend the description of the Witt vector ring as a completion to all $\oF_p$-algebras with injective Frobenius map.

The paper is self-contained and elementary. Except for the results comparing the objects that we obtain to the conventional ones, it can be read without knowing the classical theories in the title. For these we refer to \cite{T}, \cite{Har}, \cite{S}, \cite{Ha}, \cite{W1}, \cite{W2}, \cite{R}, \cite{H}, \cite{B} and the appendix of \cite{M} as far as Witt vector rings are concerned. For the de~Rham Witt theory we mention the works \cite{Ch}, \cite{I}, \cite{H}, \cite{LZ} and the references given there. More advanced topics are discussed in \cite{DK}, \cite{DLZ1}, \cite{DLZ2} for example.

Our constructions of Witt vector rings and of the de~Rham Witt complex are so simple that it is worthwhile to give the details in this introduction --- at least in the $p$-typical case and for $\Z$-algebras $R$. In case $R$ is an $\F_p$-algebra this gives a quick alternative approach to the original de~Rham Witt complex in \cite{I}. We will describe the construction of a commutative ring $E (R)$ and a differential graded ring $E \Omega^{\bfdot}_{R / \Z}$ with $E \Omega^0_{R / \Z} = E (R)$ and equipped with Frobenius and Verschiebung maps. The ring $E (R)$ is isomorphic to the $p$-typical Witt ring $W (R)$ and there is a canonical isomorphism:
\begin{equation} \label{eq:i1}
E \Omega^{\bfdot}_{R / \Z} \cong W \Omega^{\bfdot}_{R / \Z} := \varprojlim_n W_n \Omega^{\bfdot}_{R / \Z} \; .
\end{equation}
Here $W_n \Omega^{\bfdot}_{R / \Z}$ is the relative de~Rham Witt complex of Chatzistamatiou \cite{Ch} for the truncation set $S = \{ 1 , p , \ldots , p^{n-1} \}$. For an $\F_p$-algebra $R$ it is isomorphic to the complex $W_n \Omega^{\bfdot}_R$ of Bloch, Deligne and Illusie \cite{I}.\footnote{We use the notation $E (R)$ (with the name Ernst Witt in mind) in order to express the fact that $E (R)$ while being isomorphic to the usual Witt ring $W (R)$ is constructed in a quite different way.}

Let us now spell out the details of  the construction of $E (R)$ and of $E \Omega^{\bfdot}_{R / \Z}$. For a discrete ring $A$ we view $A^{\N_0}$ with the product topology as a topological ring under componentwise addition and multiplication. The ring endomorphism $F$ of $A^{\N_0}$ is defined to be the left shift
\[
F (a_0 , a_1 , \ldots) = (a_1 , a_2 , \ldots)\; .
\]
We obtain an additive self map $V$ of $A^{\N_0}$ by setting
\[
V (a_0 , a_1 , \ldots ) = p (0, a_0 , a_1 , \ldots)\; .
\]
For $a \in A$ set $\langle a \rangle = (a , a^p , a^{p^2} , \ldots) \in A^{\N_0}$. Thus we have
\[
V^n \langle a \rangle = p^n (0, \ldots , a , a^p , \ldots) \quad \text{with $a$ in $n$-th place.}
\]
For commutative $A$ consider the closed subgroup $X (A) \subset A^{\N_0}$ generated by all elements $V^n \langle a \rangle$ for $a \in A$ and $n \ge 0$. Because of the formula
\[
V^n \langle b \rangle V^m \langle c \rangle = p^n V^m \langle b^{p^{m-n}} c \rangle \quad \text{for} \; n \le m \, , \, b,c \in A
\]
$X (A)$ is actually a closed subring of $A^{\N_0}$. Similarly, for an ideal $I$ in $A$, the subgroup $X (I)$ of $X (A)$ is a closed ideal. The maps $F$ and $V$ leave $X (A)$ and $X(I)$ invariant and therefore pass to the quotient $X (A) / X (I)$.

For a commutative ring $R$ let $\Z R$ be the semigroup algebra on the multiplicative semigroup $(R , \hullet)$. It is the free $\Z$-module on symbols $[r]$ for $r \in R$ with multiplication induced by the rule $[r_1] [r_2] = [r_1 r_2]$. The natural map $\pi : \Z R \to R$ sending $[r]$ to $r$ leads to a presentation
\begin{equation} \label{eq:i1a}
0 \longrightarrow I \longrightarrow \Z R \longrightarrow R \longrightarrow 0 \; .
\end{equation}
The quotient
\begin{equation}
\label{eq:i2}
E (R) = X (\Z R) / X (I)
\end{equation}
is a topological ring which is equipped with self maps $V$ and $F$. By construction, the canonical Teichm\"uller map
\[
\langle \, \rangle : R \longrightarrow E (R) \, , \, \langle r \rangle = \langle [r] \rangle \mod X (I)
\]
is multiplicative and everything is functorial in $R$. We show that there are unique topological isomorphisms of rings $E (R) \cong W (R)$ which respect the Frobenius, Verschiebung and Teichm\"uller maps.

In our construction of the de~Rham Witt complex, for technical reasons we use the free commutative algebra $\Z [R]$ on the set $R$ instead of $\Z R$ i.e. the polynomial algebra over $\Z$ in indeterminates $t_r$ for $r \in R$. Mapping $t_r$ to $r$ gives a presentation (with a different $I$)
\[
0 \longrightarrow I \longrightarrow \Z [R] \longrightarrow R \longrightarrow 0
\]
and there is a canonical isomorphism
\begin{equation}
\label{eq:i3}
E (R) = X (\Z [R]) / X (I) \; .
\end{equation}
This is a non-trivial fact. The right hand side is equipped with self maps $F$ and $V$ and it is the target of a Teichm\"uller map given by $\langle r \rangle = \langle t_r \rangle \mod X (I)$. It follows from \eqref{eq:i3} that this map is multiplicative but it is not so easy to see this directly. We will think of $E (R)$ as the quotient \eqref{eq:i3} for the moment. \\
Set $\Omega = \Omega^{\bfdot}_{\Z [R] / \Z}$, the differential graded ring of relative differential forms of $\Z [R]$ over $\Z$ and set $\Omega_{\Q} = \Omega \otimes \Q$. We consider $\Omega^{\N_0}_{\Q}$ as a differential graded topological ring with the differential $\dl$ defined by the formula
\[
\dl ((\omega_n)) = (p^{-n} d \omega_n) \; .
\]
Note that the subring $\Omega^{\N_0}$ is  not $\dl$-invariant. Let $X^{\bfdot} (\Z [R])$ be the (topologically) closed differential graded subring of $(\Omega^{\N_0}_{\Q} , \dl)$ generated by $X (\Z [R])$. It turns out to lie in $\Omega^{\N_0}$. Moreover let $X^{\bfdot} (I , \Z [R])$ be the closed differential graded ideal of $X^{\bfdot} (\Z [R])$ generated by $X (I)$. We define
\[
E \Omega^{\bfdot}_{R / \Z} := X^{\bfdot} (\Z [R] ) / X^{\bfdot} (I, \Z [R]) \; .
\]
The maps $F$ and $V$ on $\Omega^{\N_0}$ leave $X^{\bfdot} (\Z [R])$ and $X^{\bfdot} (I, \Z [R])$ invariant and hence pass to $E \Omega^{\bfdot}_{R /\Z}$. Straightforward calculations show that all the usual relations of de~Rham Witt theory hold between $F$ and $V$. Using a basic result about Chatzistamatiou's de~Rham Witt complex based on the work of Langer and Zink it is not difficult to establish the isomorphism \eqref{eq:i1} mentioned above.



Motivated by topological Hochschild homology, Hesselholt \cite{H} defined an absolute de~Rham Witt complex $W \Omega^{\bfdot}_R$ for rings. It is the inital complex in his category of "Witt-complexes". There is a natural surjective map of Witt-complexes
\[
W \Omega^{\bfdot}_R \twoheadrightarrow E \Omega^{\bfdot}_{R / \Z}
\]
but it is not injective. For example $E \Omega^1_{\Z / \Z} = 0$ whereas $W \Omega^1_{\Z}$ is non-zero.

In \cite{Ch} Chatzistamatiou defined his relative de~Rham Witt complex for every $R_0$-algebra $R$ using quotients of Hesselholt's absolute de~Rham Witt complexes. More generally than \eqref{eq:i1}, we show that Chatzistamatiou's and our relative de~Rham Witt complexes are canonically isomorphic. 

We now turn to the second theme of the paper. We will state the results using the familiar $W$-notation but the proofs use the $E$-description of the Witt vector rings. As above, let $\Z R$ be the semigroup algebra of $(R , \hullet)$ and $I$ the kernel of the natural homomorphism $\Z R  \to R$.
In our previous note \cite{CD} we observed that the ring $W (R)$ of $p$-typical Witt vectors of a perfect $\F_p$-algebra $R$ is canonically isomorphic to a completion
\begin{equation}\label{eq:5n}
W (R) = \varprojlim_n \Z R / I^n \; .
\end{equation}
In fact we have $W_n (R)= \Z R / I^n$ canonically. In the last part of the present paper we generalize this result in two directions. Firstly, for a perfect $\F_p$-algebra $R$ we may work with more general presentations of the form
\[
0 \longrightarrow I \longrightarrow A \longrightarrow R \longrightarrow 0
\]
where $A$ has no $p$-torsion and is equipped with an isomorphism lifting the Frobenius isomorphism of $R$. Then Theorem \ref{t55} asserts that we have a natural isomorphism
\[
W (R) = \varprojlim_{\nu} A / I^{\nu} \; .
\]
This generalizes both \eqref{eq:5n} and a well known classical result, c.f. Corollary \ref{t65nn}. Secondly we consider non perfect $\F_p$-algebras $R$. The multiplicative Teichm\"uller map $R \hookrightarrow W (R)$ induces ring homomorphisms $\Z R \to W (R)$ and
\[
\Z R / I^n \to W_n (R)\; .
\]
In the non-perfect case the latter map is neither injective nor surjective in general. We now restrict attention to $\F_p$-algebras $R$ whose Frobenius endomorphism is injective e.g.~reduced algebras. Let $\phi$ be the lifted Frobenius endomorphism of $\Z R$ determined by the formula $\phi ([r]) = [r^p]$ for $r \in R$. For $n \ge 1$ we set
\[
I_n = \phi^{1-n} (I^n) = \{ a \in \Z R \mid \phi^{n-1} (a) \in I^n \} \; .
\]
This is an ideal in $\Z R$ which can be shown to have the alternative description
\[
I_n = \bigcup_{\nu \ge 1} \phi^{-\nu} (I^n) \; .
\]
We prove that for any $\F_p$-algebra $R$ with injective Frobenius map, the natural map $\Z R \to W_n (R)$ induces an inclusion of rings
\[
\beta_n : \Z R / I_n \hookrightarrow W_n (R) \; .
\]
If $R$ contains $\oF_p$ then the map $\beta_n$ is an isomorphism. Hence there is a topological isomorphism
\[
W (R) = \varprojlim_n \Z R / I_n \; .
\]

The proof of the second assertion requires a result about polynomials which is proved in the Appendix by Umberto Zannier.

The contents of the individual sections are the following: Sections \ref{sec:1} and \ref{sec:2} develop the basics of our approach to Witt vector rings and explain how the usual theory is recovered. We also remark on a way to extend our version of Witt vector theory from commutative to non-commutative rings. Section \ref{sec:3} gives the simple new approach to relative de~Rham Witt theory in the same spirit. At the end we make some remarks on the overconvergent theory. Sections \ref{sec:4} and \ref{sec:5} return to Witt vector theory from our perspective explaining how some further required properties are proved in our setup. In section \ref{sec:6} we compare the rings $E (R)$ for perfect $\F_p$-algebras $R$ directly to $\varprojlim_n \Z R / I^n$ without the detour over ordinary Witt vector rings. We also discuss more general presentations. Finally section \ref{sec:7} is concerned with non-perfect $\F_p$-algebras as above.

We would like to thank Lars Hesselholt for a helpful exchange and Umberto Zannier for contributing the appendix. We are grateful to Jim Borger for interesting comments and for pointing out \cite{Ch} to us after receiving the first version of the paper.

\section{Preliminaries} \label{sec:1}
In this section we introduce some notation following \cite{H} and state a number of elementary facts which are used in the sequel. With a few exceptions, in this paper all rings are commutative and possibly non-unital. Statements hold both in the category of rings and the category of unital rings.

A truncation set $S$ is a non-empty subset of the positive integers $\N$ such that for $n \in S$ every positive divisor $d$ of $n$ is in $S$ as well. In particular $1 \in S$. We will view a truncation set $S$ as a partially ordered set under divisibility and write $n \mid m$ if $n$ divides $m$ and $n \| m$ if in addition $n \neq m$. For $n \in S$ the set $S / n = \{ \nu \in \N \mid \nu n \in S \}$ is a truncation set as well. We have $(S / n) / m = S / nm$. The greatest common divisor of $n ,m \in \N$ is denoted by $(n, m)$ and the least common multiple by $[n,m]$.

For a ring $A$ which for the moment may be non-commutative, consider $A^S$ as a ring under componentwise addition and multiplication. The family of rings $(A^S)$ for varying truncation sets $S$ carries the following extra structures:

(1) For every $n \in S$ the surjective Frobenius homomorphism of rings
\[
F_n : A^S \longrightarrow A^{S/n} \quad \text{with} \; F_n ((a_{\nu})_{\nu \in S}) = (a_{\nu n})_{\nu \in S / n}
\]
(2) For every $n \in S$ the additive Verschiebung
\[
V_n : A^{S/n} \longrightarrow A^S \quad \text{with} \; V_n ((a_{\nu})_{\nu \in S / n}) = n (\delta_{n \mid \nu} a_{\nu / n})_{\nu \in S} \; .
\]
Here $\delta_{n \mid \nu} = 1$ if $n \mid \nu$ and $\delta_{n \mid \nu} = 0$ if $n \nmid \nu$.

(3) For truncation sets $T \subset S$, the surjective projection homomorphism \\
$A^S \to A^T$.

The proofs of the following assertions are straightforward.

\begin{proposition} \label{t1} 
a) For $nm \in S$ we have
\[
F_n F_m = F_{nm} : A^S \longrightarrow A^{S/nm} \quad \text{and} \quad V_n V_m = V_{nm} : A^{S/nm} \longrightarrow A^S\; .
\]
b) For $n,m \in S$ set $n' = n / (n,m)$ and $m' = m / (n,m)$. Then we have
\[
F_m V_n = (n,m) V_{n'} F_{m'} : A^{S/n} \longrightarrow A^{S/m} \; .
\]
c) For $n \in S, \bfa \in A^S , \bfb \in A^{S/n}$ and $\bfc \in A^{S/m}$
\[
V_n (F_n (\bfa) \bfb) = \bfa V_n (\bfb) \quad \text{and} \quad V_n (\bfa F_n (\bfb)) = V_n (\bfa) \bfb \; .
\]
Hence
\[
V_n (\bfb) V_m (\bfc) = (n,m) V_{[n,m]} (F_{m'} (\bfb) F_{n'} (\bfc))
\]
and for $l \in \N$ also
\[
V_n (\bfb)^l = n^{l-1} V_n (\bfb^l) \; .
\]
d) For $n \in T \subset S$ the following diagrams commute
\[
\xymatrix{
A^{S/n} \ar[r]^{V_n} \ar[d] & A^S \ar[d] \\
A^{T/n} \ar[r]^{V_n}  & A^T
} \quad \text{and} \quad
\xymatrix{
A^S \ar[r]^{F_n} \ar[d] & A^{S/n} \ar[d] \\
A^T \ar[r]^{F_n} & A^{T/n}
}
\]
For $n \in S \setminus T$ the composition $A^{S/n} \xrightarrow{V_n} A^S \to A^T$ is the zero map. 
\end{proposition}

We give $A$ the discrete topology and $A^S$ the product topology. Then $A^S$ is a Hausdorff topological ring and the maps $F_n$ and $V_n$ are continuous.
A sequence $\bfx^{(\nu)}$in $A^S$ converges to $\bfx \in A^S$ if and only if for each $n \in S$ the sequence $x^{(\nu)}_n$ in $A$ becomes stationary with value $x_n$. In particular a series $\sum_{\nu} \bfa^{(\nu)}$ converges in $A^S$ if and only if for each $n \in S$ almost all $a^{(\nu)}_n$ are zero. Hence convergence and the value of the limit do not depend on the order of summation. For $\bfa \in A^{S/n}$ we have $(V_n (\bfa))_m = 0$ if $n \nmid m$. Hence any series of the form $\sum_{n \in S} V_n (\bfa^{(n)})$ with $\bfa^{(n)} \in A^{S/n}$ is unconditionally convergent in $A^S$.

From now on let $A$ be a commutative ring and $I$ an ideal in $A$. We view $I$ as a ring too. We write
\[
V_n \langle a \rangle := V_n ((a^{\nu})_{\nu \in S/n}) = n (\delta_{n \mid \nu} a^{\nu / n})_{\nu \in S} \quad \text{in} \; A^S\; .
\]

\begin{defprop}
\label{tx2}
Let $X_S (A)$ be the closed additive subgroup of $A^S$ generated by the elements $V_n \langle a\rangle$ for $n \in S$ and $a \in A$. Then the following assertions hold:\\
1) $X_S (A)$ is a closed subring of $A^S$ and $X_S (I)$ is a closed ideal in $X_S (A)$. \\
2) The continuous maps $V_m$ and $F_m$ restrict to continuous maps
\[
V_m : X_{S / m} (A) \longrightarrow X_S (A) \quad \text{and} \quad F_m : X_S (A) \longrightarrow X_{S/m} (A) \; .
\]
Moreover there is a multiplicative homomorphism
\[
\langle \, \rangle_S : A \longrightarrow A^S \; \text{given by} \; \langle a \rangle_S = V_1 \langle a \rangle = (a^{\nu})_{\nu \in S} \; .
\]
We have $V_n \langle a \rangle = V_n (\langle a \rangle_{S/n})$ in $A^S$.\\
3) For a finite truncation set $T \subset S$ the projection $A^S \to A^T$ restricts to a continuous surjective homomorphism $X_S (A) \to X_T (A)$. The induced map
\[
X_S (A) = \varprojlim_T X_T (A)
\]
is an isomorphism of topological rings. Here $T$ runs over the directed poset under inclusion of all finite truncation subsets of $S$.
\end{defprop}

It it clear that the analogues of all relations in Proposition \ref{t1} hold for the maps $F_m , V_n$ between the rings $X_S (A)$ resp. $X_S (I)$.

\begin{proof}
Using Proposition \ref{t1} we obtain
\begin{equation}
\label{eq:1}
V_n \langle b\rangle V_m \langle c\rangle = (n,m) V_{[n,m]} \langle b^{m'} c^{n'} \rangle \quad \text{for} \; b, c \in A \, , \, n,m \in S \; .
\end{equation}
Recall here that $n' = n / (n,m)$ and $m' = m / (n,m)$. This implies 1). Similarly we find
\begin{equation}
\label{eq:2}
V_m (V_n \langle a\rangle) = V_{mn} \langle a\rangle \quad \text{for} \; a \in A \; \text{and} \; nm \in S
\end{equation}
and
\begin{equation}
\label{eq:3}
F_m (V_n \langle a \rangle) = (n,m) V_{n'} \langle a^{m'} \rangle \quad \text{for $a \in A$ and $n,m \in S$} \; .
\end{equation}
Hence we get 2). The first part of assertion 3) holds because $V_n (\langle a \rangle_{S/n})$ projects to $V_n (\langle a \rangle_{T/n})$ if $n \in T \subset S$. As for the second part, note that the projective limit topology of
\[
\varprojlim_{T \subset S} X_T (A) \subset \varprojlim_{T \subset S} A^T = A^S
\]
agrees with the subspace topology in $A^S$. Moreover by the surjectivity assertion the continuous injection
\[
i : X_S (A) \longrightarrow \varprojlim_{T \subset S} X_T (A) \subset A^S
\]
has dense image. Since $X_S (A)$ is closed in $A^S$ we are done.
\end{proof}

We are particularly interested in the unconditionally convergent series
\begin{equation}
\label{eq:6a}
\Psi_S (\bfa) = \sum_{n \in S} V_n \langle a_n \rangle \quad \text{for} \; \bfa = (a_n) \in A^S \; .
\end{equation}
It takes values in $X_S (A)$ and its $m$-th component for $m \in S$ is given by
\begin{equation}
\label{eq:5}
\Psi_S (\bfa)_m = \sum_{n \mid m} n a^{m/n}_n \; .
\end{equation}

Note that $X_S (A)$ is the closed subgroup generated by the image of $\Psi_S$. For truncation sets $T \subset S$ the following diagram commutes where the vertical maps are projections:
\begin{equation}
\label{eq:6}
\xymatrix{
A^S \ar[r]^{\Psi_S} \ar[d] & X_S (A) \ar[d] \\
A^T \ar[r]^{\Psi_T} & X_T (A) \; .
}
\end{equation}

This follows from Proposition \ref{t1}, d). The following fact will be needed in section \ref{sec:3}. 

\begin{lemma} \label{tn13}
Let $S$ be finite and set $\Q_S = \Z [1 / p$ for $p \in S]$. Then for any ring $A$, the inclusion $X_S (A) \subset A^S$ induces an isomorphism $X_S (A) \otimes \Q_S = A^S \otimes \Q_S$.
\end{lemma}

\begin{proof}
For an integer $n$ let $\nu_p (n)$ be its $p$-valuation and let $\nu (n) = \sum_p \nu_p (n)$ be the number of prime factors of $n$ counted with their multiplicities. Since $X_S (A) \otimes \Q_S$ is the $\Q_S$-submodule of $(A \otimes \Q_S)^S = A^S \otimes \Q_S$ generated by the elements $V_n \langle a \rangle \otimes 1$ for $n \in S$ and $a \in A$ it suffices to prove the following assertion: For $\bfx \in A^S$ and any integer $N$ with prime divisors in $S$, there exist $a_n \in A$ and $q_n \in \Q_S$ for $n \in S$ such that we have
\[
\sum_{n \in S} V_n \langle a_n \rangle \otimes q_n = \bfx \otimes N^{-1} \quad \text{in} \; A^S \otimes \Q_S \; .
\]
Taking $q_n = n^{-1} N^{-1}$ it therefore suffices to solve the following equations for the $a_n$'s 
\begin{equation} \label{eq:11.5}
\sum_{n \tei m} a^{m/n}_n = x_m \quad \text{in $A$ for all} \; m \in S \; .
\end{equation}
Set $a_1 = x_1$ and given $i \ge 0$ assume that for $n \in S$ with $\nu (n) \le i$, elements $a_n \in A$ have been determined such that \eqref{eq:11.5} holds for all $m \in S$ with $\nu (m) \le i$. For $m \in S$ with $\nu (m) = i +1$, equation \eqref{eq:11.5} holds with the $a_n$'s already found and with
\[
a_m := x_m - \sum_{n \| m} a^{m/n}_n \; .
\]
\end{proof}

We end the section with a small topological lemma which will be needed later.

\begin{lemma} \label{t23}
Let $I$ be a countable directed poset and consider a projective system of exact sequences of abelian groups with the discrete topologies
\[
0 \longrightarrow N_i \longrightarrow G_i \xrightarrow{\pi_i} H_i \longrightarrow 0 \quad \text{for} \; i \in I \; .
\]
Assume that the transition maps $N_i \to N_j$ are surjective for $i \ge j$. Setting $N = \varprojlim_i N_i$ with the projective limit topology etc. we get an exact sequence
\[
0 \longrightarrow N \longrightarrow G \xrightarrow{\pi} H \longrightarrow 0
\]
where the continuous map $\pi$ is open. In particular $\pi$ induces a topological isomorphism $G / N \silo H$.
\end{lemma}

\begin{proof}
It is enough to show that $\pi (U)$ is open for every open subgroup $U = U_J$ in $G$ of the form
\[
U_J = \{ g = (g_i) \in G \mid g_i = 0 \; \text{for}\; i \in J \} \; ,
\]
where $J \subset I$ is a finite subset with a maximal element. Namely, since $I$ is directed, $(U_J)$ is a fundamental system of neighborhoods of $0 \in G$. It suffices to show that
\begin{equation}
\label{eq:10n}
\pi (U_J) = \{ h = (h_i) \in H \mid h_i = 0 \; \text{for} \; i \in J \}
\end{equation}
since the right hand side is open. The inclusion "$\subset$" being trivial, consider $h \in H$ with $h_i = 0$ for $i \in J$ and choose some $\tg \in G$ with $\pi (\tg) = h$. Then $\tg_i \in N_i$ for $i \in J$. Let $i_0 \in J$ be a maximal element. Since $I$ is countable and the transition maps of $(N_i)_{i \in I}$ are surjective, the projection map $N \to N_{i_0}$ is surjective and hence there is an element $n \in N$ with $n_{i_0} = \tg_{i_0}$. It follows that $n_i = \tg_i$ for all $i \in J$. For the element $g = \tg - n \in U_J$ we therefore have $\pi (g) = h$.
\end{proof}
\section{A new approach to Witt vector rings} \label{sec:2}
For a ring $R$ we introduce new rings which are naturally isomorphic to the truncated Witt vector rings of $R$. In our approach the ring operations are obvious but not the natural bijection to a power of $R$. 

Consider the presentation \eqref{eq:i1a}
\[
0 \longrightarrow I \longrightarrow \Z R \xrightarrow{\pi} R \longrightarrow 0
\]
and recall the notations and results of Proposition \ref{tx2}.

\begin{definition} \label{tx3}
For a ring $R$ and any truncation set $S$ we set
\[
E_S (R) = X_S (\Z R) / X_S (I) \; .
\]
These are Hausdorff topological rings which are related by functorial continuous Frobenius, Verschiebung and projection operators.
\end{definition}
All the relations in Proposition \ref{t1} also hold if the rings $A^S$ there are replaced with $E_S (R)$.

The projection $\pi : \Z R \to R$ in \eqref{eq:i1a} is split by the map $r \to [r]$. We define the multiplicative (Teichm\"uller) map $\langle \rangle_S : R \to E_S (R)$ as the composition:
\begin{equation} \label{eq:4a}
\langle \rangle_S : R \xrightarrow{[\,]} \Z R \xrightarrow{\langle \rangle_S} X_S (\Z R) \longrightarrow E_S (R) \; .
\end{equation}
We will soon see that the rings $E_S (R)$ with these extra structures are functorially isomorphic to the truncated Witt vector rings $W_S (R)$. For later purposes it is sometimes necessary to work with other presentations of $R$. We therefore develop the theory in this more general context.

\begin{definition} \label{tx4}
For a presentation
\begin{equation}
\label{eq:4}
0 \longrightarrow I  \longrightarrow A \xrightarrow{\pi} R \longrightarrow 0
\end{equation}
where $A$ is a ring without $S$-torsion, set
\[
E_S (R,\pi) = X_S (A) / X_S (I) \; .
\]
\end{definition}

It is a Hausdorff topological ring equipped with continuous Frobenius, Verschiebung and projection operators satisfying the analogues of the relations in Proposition \ref{t1}. It is true but non-trivial that the multiplicative map
\[
\langle \rangle_S : A \longrightarrow X_S (A) \longrightarrow E_S (R,\pi)
\]
descends to a multiplicative map $\langle \rangle_S : R = A / I \to E_S (R,\pi)$. In fact the rings $E_S (R,\pi)$ and $E_S (R)$ are canonically isomorphic, compatibly with $V_n , F_m$ and $\langle \rangle_S$. We will see this in Corollary \ref{j7}, c) below.

For a presentation \eqref{eq:4} there is an exact sequence of projective systems where $T$ runs over the countable poset of finite truncation subsets of $S$
\[
0 \longrightarrow (X_T (I)) \longrightarrow (X_T (A)) \longrightarrow (E_T (R,\pi)) \longrightarrow 0 \; .
\]
Since $X_{T'} (I) \to X_T (I)$ is surjective for $T \subset T' \subset S$ it follows from Lemma \ref{t23} that as topological rings we have

\begin{equation}
\label{eq:8}
E_S (R,\pi) = \varprojlim_T E_T (R,\pi) \quad \text{and in particular} \quad E_S (R) = \varprojlim_T E_T (R) \; .
\end{equation}

To analyze the additive group of $E_S(R, \pi)$ and the map $ \Psi_S$ we use an argument based on the ideas of Witt in \cite{W2}. Let $A$ be any commutative ring. An easy recursive argument shows that the following map is bijective and in fact a homeomorphism for the $t$-adic topology on the right
\begin{equation} \label{eq:x11}
\lambda : A^{\N} \longrightarrow \hat{U} (A) := 1 + tA \llbracket t \rrbracket \; , \; (a_n) \longmapsto \prod^{\infty}_{n=1} (1 - a_n t^n) \; .
\end{equation}
Here, if $A$ has no unit element, we embed it into a unital ring $\tilde{A}$ and let $1$ be the unit element of $\tilde{A}$. We view $\hU (A)$ and $tA \llbracket t \rrbracket$ as abelian groups under multiplication resp. addition. The map
\[
\mu_A : \hat{U} (A) \longrightarrow tA[[t]] \equiv A^{\N} \; , \; \mu_A (P) = - tP' / P
\]
is a continuous homomorphism. Writing $ \Psi = \Psi_\N$ we find
\begin{align} \label{eq:18nn}
\mu_A \prod^{\infty}_{n=1} (1 - a_n t^n) & = \sum^{\infty}_{n=1} n \frac{a_n t^n}{1 - a_n t^n} = \sum_{n, \nu} n a^{\nu}_n t^{n\nu} \\
& = \sum_m \big( \sum_{n \mid m} n a^{m/n}_n \big) t^m \equiv \Psi ((a_n)) \; . \nonumber
\end{align}

Hence $ \Psi = \mu_A \circ \lambda$ and $ \mu_A$ takes values in Im$ \Psi \subset X(A) := X_\N(A)$.

In $\hat{U}(A)$ consider the closed
subgroup $H_S$ generated by the elements $1 - at^n$ for $a \in A, n \in \N \setminus S$ and the quotient $\hat{U}_S (A):= \hat{U}(A)/H_S$. Viewing $A^S$ as a quotient of $A^\N$, the map $\lambda$ induces a homeomorphism $\lambda_S : A^S \to \hat{U}_S (A)$. \\
Formula \eqref{eq:18nn} implies that $\mu_A(1-at^n)= V_n \langle a \rangle$ for $n \in \N$ and $a \in A$. For $n \in \N \setminus S$ the element $V_n \langle a \rangle$ is mapped to zero under the projection $X(A) \to X_S(A)$ by Proposition \ref{t1} d). Therefore $ \mu_A$ descends to a continuous homomorphism $$\mu_{A,S}:\hat{U}_S(A) \to X_S(A).$$
The equation $ \Psi = \mu_A \circ \lambda$ implies that $ \Psi_S = \mu_{A,S} \circ \lambda_S$.

Let $I$ be an ideal in the commutative ring $A$. Then the kernel of the natural surjection $\hat{U}_S(A) \to \hat{U}_S(A/I)$ is exactly $\hat{U}_S(I)$. Therefore $A \mapsto \hat{U}_S(A)$ is an exact functor to abelian groups.
\begin{proposition}
\label{t2}
a) For any ring $A$, the image of $\Psi_S$ is a subgroup of $X_S (A)$.\\
b) For a ring $A$ without $S$-torsion the map $\Psi_S : A^S \to X_S (A)$ defined in \eqref{eq:6a} is a homeomorphism and the map $\mu_{A,S}: \hat{U}_S(A) \to X_S(A)$ is an isomorphism of topological groups.
\end{proposition}
\begin{proof}
a) Because of the relation $\Psi_S = \mu_{A,S} \verk \lambda_S$ where $ \lambda_S$ is bijective, the maps $ \Psi_S$ and $\mu_{A,S}$ have the same image in $X_S(A)$. Since $\mu_{A,S}$ is a homomorphism, this image is a subgroup.\\
b) First we prove that $\Psi_S$ is bijective for finite $S$. By a) the image of $\Psi_S$ is a subgroup of $X_S (A)$ which contains all the elements $V_n \langle a\rangle$ for $n \in S$ and $a \in A$. Since $X_S (A)$ is the smallest subgroup of $A^S$ with this property (because $S$ is finite), it follows that $\Psi_S$ is surjective. For an integer $n$ let $\nu(n)$ be the number of prime factors of $n$ counted with their multiplicities. Assume that in formula \eqref{eq:5} we have $\Psi_S (\bfa) = 0$ for some $\bfa \in A^S$. For $m = 1$ we get $a_1 = 0$. Assume that $a_n = 0$ for all $n$ with $\nu (n) \le i$ has been shown for some $i \ge 0$. Let $\nu (m) = i+1$. Then $0 = \Psi_S (\bfa)_m = m a_m$ by the induction hypotheses. Since $A$ has no $S$-torsion it follows that $a_m = 0$. Thus $\bfa = 0$ and hence $\Psi_S$ is injective. \\
For finite truncation sets $T$ the bijections $\Psi_T : A^T \to X_T (A)$ are trivially even homeomorphisms. Now let $S$ be arbitrary. Using Proposition \ref{tx2}, 3) and diagram \eqref{eq:6} we find that
\[
\Psi_S = \varprojlim_{T \subset S} \Psi_T : A^S = \varprojlim_{T \subset S} A^T \longrightarrow X_S (A) = \varprojlim_{T \subset S} X_T (A) \; .
\]
Hence $\Psi_S$ is a homeomorphism. \end{proof}

\begin{theorem}\label{jc}
Let $R$ be a ring and $0 \to I \to A \xrightarrow{\pi} R \to 0$ a presentation where $A$ has no $S$-torsion. Then there are a unique homeomorphism $\psi_S : R^S \to E_S(R, \pi)$ and a unique isomorphism $\nu_{R,S}: \hat{U}_S(R) \to E_S(R, \pi)$ of topological groups such that the following diagrams commute
\begin{equation}
\label{eq:9}
\xymatrix{
A^S \ar[r]^-{\Psi_S} \ar[d]_{\pi^S} & X_S (A) \ar[d] && \hat{U}_S(A) \ar[r]^-{\mu_{A,S}} \ar[d] & X_S (A) \ar[d] \\
R^S \ar[r]^-{\psi_S}& E_S (R,\pi)  &&\hat{U}_S(R) \ar[r]^-{\nu_{R,S}}& E_S (R,\pi)\;.
}
\end{equation}
\end{theorem}
\begin{proof}
The uniqueness is clear. The topological isomorphisms $\mu_{A,S}$ and $\mu_{I,S}$ of Proposition \ref{t2} induce the topological isomorphism
$$\nu_{R,S}: \hat{U}_S(R) = \hat{U}_S(A)/\hat{U}_S(I) \longrightarrow X_S(A)/X_S(I) = E_S(R, \pi) \; .$$
Composing with the homeomorphism $\lambda_S: R^S \to \hat{U}_S(R)$ we obtain $ \psi_S$ as $\nu_{R,S}\circ \lambda_S$.\end{proof}

We can now introduce the Teichm\"{u}ller character $\langle \rangle_S : R \to E_S(R, \pi)$ by defining $ \langle r \rangle _S := \psi_S((r,0, \ldots ))$ for $r \in R$. For $a \in A$ we have
\begin{equation} \label{j5} \langle \pi (a) \rangle_S = \psi_S ((\pi (a), 0, \ldots)) = \langle a \rangle_S \mod X_S(I)\end{equation}
Hence the resulting map $\langle \, \rangle_S : R \to E_S(R, \pi)$ is multiplicative. Using this Teichm\"{u}ller character it follows from the first diagram in Theorem \ref{jc} that the map $ \psi_S :R^S \to E_S(R, \pi)$ has the following description
\begin{equation} \label{j6} \psi_S(r) = \sum_{n \in S}V_n\langle r\rangle 
\end{equation}
Here, as before, we set $V_n \langle r\rangle = V_n(\langle r\rangle_{S/n})$.

The association $R \mapsto E_S(R)$ from Definition \ref{tx3} is functorial by construction. We now discuss functoriality for the rings $E_S(R, \pi)$ and obtain in particular a functorial identification $E_S(R, \pi) \cong E_S(R)$. Consider two $S$-torsion free presentations $0 \to I_\nu \to A_\nu \xrightarrow{\pi_\nu} R_\nu \to 0$ for $ \nu = 1,2$. A ring homomorphism $\alpha : R_1 \to R_2$ induces a continuous map $E_S( \alpha): E_S(R_1,\pi_1) \to E_S(R_2, \pi_2)$ by the formula $E_S( \alpha) \circ \psi_S = \psi_S \circ \alpha^S$. It is clear that $E_S( \alpha_2 \circ \alpha_1) = E_S( \alpha_2) \circ E_S( \alpha_1)$ for maps $R_1 \xrightarrow{\alpha_1}R_2 \xrightarrow{\alpha_2}R_3$ and torsion-free presentations of the $R_i$.

\begin{cor}\label{j7} a) The map $E_S( \alpha)$ is a continuous homomorphism of rings which commutes with $ \langle \, \rangle_S$. Moreover $E_S( \alpha)$ and $E_{S/n}( \alpha)$ commute with $F_n$ and $V_n$ in the evident sense for $n \in S$.\\
b) The map $E_S( \alpha)$ can be characterized as the unique continuous additive map $E_S(R_1, \pi_1) \to E_S(R_2, \pi_2)$ which sends $V_n \langle r \rangle$ to $V_n \langle \alpha(r) \rangle$ for $r \in R_1$, $n \in S$.\\
c) For any presentation $0 \to I \to A \xrightarrow{\pi} R \to 0$ the map $E_S( \id)$ gives an identification $E(R, \pi) \cong E_S(R)$. \end{cor}

\begin{proof} a) We first assume that $ \alpha$ lifts to a map of presentations, i.e.~to a commutative diagram
$$\xymatrix{
0 \ar[r] & I_1 \ar[r] \ar[d]^{\tilde{\alpha}} & A_1 \ar[r]^{\pi_1} \ar[d]^{\tilde{\alpha}} & R_1 \ar[r] \ar[d]^{\alpha} & 0 \\
0 \ar[r] & I_2 \ar[r] & A_2 \ar[r]^{\pi_2} & R_2 \ar[r] & 0 \; .
}$$

Then $ \tilde{\alpha}$ induces a homomorphism of rings
$$E_S( \talpha) : E_S(R_1, \pi_1) = X_S(A_1)/X_S(I_1) \to X_S(A_2)/X_S(I_2) = E_S(R_2, \pi_2)$$ which is compatible with $V_n$ and $F_n$ and by (\ref{j5}) also with the Teichm\"{u}ller character. Using formula (\ref{j6}) it follows that $E_S( \talpha) \circ \psi_S = \psi_S \circ \alpha^S$ and therefore $E_S( \alpha) = E_S( \talpha)$. Hence $E_S( \alpha)$ has the desired properties.\\
In the general case choose a presentation $0 \to I \to A \xrightarrow{\pi} R_1 \to 0$ where $A$ is a polynomial algebra over $\Z$. The identity id$_{R_1}$ on $R_1$ lifts to a map from this free presentation to the presentation given by $\pi_1$. Therefore the bijection $E_S( \id_{R_1}): E_S(R_1, \pi) \to E_S(R_1, \pi_1)$ is a ring isomorphism. Similarly $\alpha :R_1 \to R_2$ lifts to a morphism from the free presentation by $\pi$ to the presentation by $\pi_2$ and therefore $E_S(\alpha): E_S(R_1, \pi) \to E_S(R_2, \pi_2)$ is a homomorphism of rings. The composition $E_S(\alpha) \circ E_S(\id_{R_1})^{-1}$ of these maps being the map $E_S(\alpha): E_S(R_1, \pi_1) \to E_S(R_2, \pi_2)$ it follows that the latter is a ring homomorphism as well. Compatibility with $ \langle \, \rangle$, $F_n$, $V_n$ follows in the same way.\\
b) The map $E_S( \alpha)$ has the stated properties, e.g. by a). The uniqueness assertion follows from the explicit formula \eqref{j6}.\\
c) is a consequence of a).
\end{proof}

In the following, given $S$-torsion free presentations of a ring $R$ by $\pi_1$ and $\pi_2$, we will view the topological ring isomorphism $E_S( \id_R): E_S(R, \pi_1) \to E_S (R, \pi_2)$ as an identification. Then $ \langle \, \rangle$, $F_n$, $V_n$ and the parametrization $\psi_S$ on both sides are identified as well.

In Theorem \ref{jc} we have seen that $E(R) \cong \hat{U}(R)$ as an abelian group. We can use this isomorphism to transport $F_m$ and $V_m$ to continuous group endomorphisms of $\hat{U} (R)$. For an explicit description, in view of the homeomorphism \eqref{eq:x11}, it suffices to describe $F_m$ and $V_m$ on the elements $1 - rt^n$ for $r \in R , n \ge 1$. By the next proposition one obtains the same formulas for $F_n$ and $V_n$ on $\hU(R)$ as the ones defined by Witt in \cite{W2}.

\begin{proposition}
\label{t3n}
Setting $n' = n / (n,m)$ and $m' = m/ (n,m)$ for $n,m \ge 1$ we have
\begin{equation} \label{eq:21a}
F_m (1 - rt^n) = (1 - r^{m'} t^{n'})^{(n,m)} \quad \text{and} \quad V_m (1 - rt^n) = 1 - rt^{nm} \; .
\end{equation}
\end{proposition}

\begin{proof}
Consider the isomorphism $\nu_R :\hU (R) \to E(R)$ from Theorem \ref{jc}  for the presentation \eqref{eq:i1a} and the corresponding commutative diagram
\[
\xymatrix{
\hU (\Z R) \ar[r]^{\mu_{\Z R}} \ar[d] & X (\Z R) \ar[d] \\
\hU (R) \ar[r]^{\nu_R} & E (R) \; .
}
\]
Using the explicit formula \eqref{eq:18nn} we find
\[
\mu_{\Z R} (1 - a^m t) = F_m \mu_{\Z R} (1 - at) \quad \text{and} \quad \mu_{\Z R} (1 - at^m) = V_m \mu_{\Z R} (1 - at)
\]
for $a \in \Z R$. Hence we obtain formula \eqref{eq:21a} for $n = 1$. The general case follows using Proposition \ref{t1}.
\end{proof}

We draw a couple of useful consequences which are well known in the usual Witt vector theory.
\begin{cor}
\label{t27n}
For a truncation set $S$ and $n \in S$ the set $T = \{ m \in S \mid n \nmid m \}$ is a truncation set. For every ring $R$ the following sequence of additive maps is exact
\[
0 \longrightarrow E_{S/n} (R) \xrightarrow{V_n} E_S (R)\xrightarrow{\proj} E_T (R) \longrightarrow 0 \; .
\]
\end{cor}

\begin{cor} \label{t28nn}
For an exact sequence of possibly non-unital rings
\[
\ldots \longrightarrow R_{i-1} \xrightarrow{\alpha_{i-1}} R_i \xrightarrow{\alpha_i} R_{i+1} \longrightarrow \ldots
\]
the sequence
\begin{equation}
\label{eq:26n}
\ldots \longrightarrow E_S (R_{i-1}) \longrightarrow E_S (R_i) \longrightarrow E_S (R_{i+1}) \longrightarrow \ldots
\end{equation}
is exact.
\end{cor}

\begin{proof}
Via the maps $\psi_S$ the complex \eqref{eq:26n} gets identified with the following sequence of pointed (by $0= (0,0, \ldots)$'s) sets
\[
\ldots \longrightarrow R^S_{i+1} \xrightarrow{\alpha^S_{i-1}} R^S_i \xrightarrow{\alpha^S_i} R^S_{i+1} \longrightarrow \ldots
\]
\end{proof}

\begin{cor}
\label{t29-n23}
Assume that $l$-multiplication by some prime number $l$ is invertible on a ring $R$. Then $l$-multiplication is invertible on $E_S (R)$ as well.
\end{cor}

\begin{proof}
First assume that $R$ is unital. Then $l$ is a unit in $R$ by assumption. As an additive group, $E (R)$ is isomorphic to the multiplicative group $\hU (R) = 1 + tR \llbracket t \rrbracket$ by Theorem \ref{jc}. Given a power series $g \in \hU (R)$ the equation $h^l = g$ has a unique solution $h$ in $\hU (R)$ because $l$ is invertible in $R$. This follows by a simple inductive argument with power series coefficients. Hence $l$-multiplication is an isomorphism on $\hU (R)$ and hence also on $E (R)$. Since $E (R)$ is a unital ring, $l$ is a unit in $E (R)$. For a truncation set $S$ the ring homomorphism $E (R) \to E_S (R)$ maps units to units. Hence $l$ is a unit in $E_S (R)$ too. If $R$ does not have a unit element we embed it as an ideal into a unital ring $\tR$ in which $l$ is a unit. This is possible by assumption on $R$. We can take $\tR = R \oplus \Z [1 / l] \cdot 1$ with the evident multiplication for example. Then both $\tR$ and $\tR /R$ are unital rings on which $l$-multiplication is invertible. Hence $l$-multiplication is invertible on $E_S (\tR)$ and $E_S (\tR / R)$. Using the exact sequence from Corollary \ref{t28nn}
\[
0 \longrightarrow E_S (R) \longrightarrow E_S (\tR) \longrightarrow E_S (\tR / R) \longrightarrow 0
\]
it follows that $l$-multiplication is invertible on $E_S (R)$ as well.
\end{proof}

There is a natural ``ghost'' homomorphism of rings
\begin{equation}
\label{eq:18}
\Gh_S : E_S (R) = X_S (\Z R) / X_S (I) \longrightarrow (\Z R)^S / I^S = R^S \; .
\end{equation}

\begin{cor}
\label{t6}
For a ring $R$ without $S$-torsion, we have $E_S (R) = X_S (R)$. In particular $E_S (R)$ has no $S$-torsion as well and the ghost map $\Gh_S : E_S (R) \to R^S$ is injective.
\end{cor}
\begin{proof}
For the presentation $0 \to 0 \to R \xrightarrow{\id} R \to 0$ we have $E_S (R,\id) = X_S (R)$.
\end{proof}

Using the bijection $\psi_S$ from Theorem \ref{jc} one can introduce the usual Witt vector rings as follows.
Let $W_S (R)$ be the set $R^S$ equipped with the unique ring structure for which $\psi_S$ is an isomorphism. Using $\psi_S$ we also transport $F_n , V_m$ and the Teichm\"uller character to $W_S (R)$. This ring structure on $W_S (R)$ is functorial in rings. The ``ghost map'' $w_S : W_S (R) \to R^S$ given by the formula $w_S ((r_n))_m = \sum_{n \mid m} nr^{m/n}_n$ for $m \in S$, is a homomorphism of rings because the following diagram commutes
\begin{equation}
\label{eq:19}
\xymatrix{
R^S \ar[r]^{\overset{\psi_S}{\sim}} \ar@{=}[d] & E_S (R) \ar[d]^{\Gh_S} \\
W_S (R) \ar[r]^{w_S} & R^S
}
\end{equation}
Let $R$ be a unital ring. Then the above ring structure on $W_S(R)$ together with Frobenius, Verschiebung and Teichm\"{u}ller character is the usual one.
Namely, the Witt vector ring structure of $R^S$ is characterized as the only functorial (in unital rings $R$) ring structure on $R^S$ for which all ghost maps are homomorphisms of rings. The uniqueness is easy to see. Namely, $w_S$ is injective if $R$ has no $S$-torsion and every ring $R$ is a quotient of a ring without $\Z$-torsion. Moreover, since $\Gh_S$ and hence $w_S$ is compatible with $F_n , V_n$ and the Teichm\"uller map it follows that these are identified with the standard ones on $W_S (R)$. Our existence proof for the functorial ring structure on $R^S$ which makes $w_S$ a homomorphism does not use universal polynomials for addition and multiplication on $R^S$.

\begin{rem}
From our perspective we obtain the universal polynomials in Witt vector theory as follows. Let $A = \Z [T_s \mid s \in S]$ be the polynomial algebra with variables indexed by $S$ and set $W_S = \spec A$. For a unital ring $R$, the $R$-valued points of $W_S$ are given by
\[
W_S (R) = \Hom_{\Z-\alg} (A , R) = R^S \; .
\]
We have equipped $W_S (R)$ with functorial ring structures. Hence $W_S$ becomes a commutative ring scheme over $\Z$. Addition and multiplication on $W_S$ correspond to homomorphisms of $\Z$-algebras $\Delta_+ , \Delta_{\hullet} : A \to A \otimes_{\Z} A$. Explicitly, $\Delta_+$ and $\Delta_{\hullet}$ are the images of $\id_{A \otimes A}$ under the addition resp. multiplication map on $W_S (R) \times W_S (R)$ for $R = A \otimes_{\Z} A$
\begin{align*}
\Hom_{\Z-\alg} (A \otimes A , A \otimes A) & = W_S (A \otimes A) \times W_S (A \otimes A) \to W_S (A \otimes A) \\
& = \Hom_{\Z-\alg} (A , A \otimes A) \; .
\end{align*}
\end{rem}
We identify $A \otimes A$ with the polynomial algebra $\Z [X_n ; Y_m \mid n, m \in S]$. Then $\Delta_+$ and $\Delta_{\hullet}$ correspond to polynomials $P_s = \Delta_+ (T_s) , Q_s = \Delta_{\hullet} (T_s)$ for $s \in S$ in variables $X_n , Y_m$ for $n,m \in S$. Addition and multiplication of $(x_s)$ and $(y_s)$ in $W_S (R) = R^S$ are then given by
\[
(x_s) + (y_s) = (P_s (x_n ; y_m \mid n,m \in S)) \quad \text{and} \quad (x_s) \cdot (y_s) = (Q_s (x_n ; y_m \mid n,m \in S)) \; .
\]
Similarly there are polynomials corresponding to the additive inverse and the induced Frobenius, Verschiebung and Teichm\"uller maps.

For rings $R$ which are not necessarily unital, the ring structure on $W_S (R)$ defined above is given by the same polynomials for addition and multiplication. This is clear because any ring embeds into a unital ring.

We now discuss certain idempotents in the ring $E_S (R)$ which generalize the classical Artin--Hasse exponential. The treatment shows again that general presentations are useful.

\begin{proposition} \label{t213xn}
Let $T \subset S$ be truncation sets and assume that an integer $n \in S$ lies in $T$ if and only if all its prime divisors lie in $T$. Consider a ring $R$ on which $l$-multiplication is invertible for all prime numbers $l \in S \setminus T$.\\
(1) The following product of commuting idempotent endomorphisms of the ring $E_S (R)$
\[
\varepsilon_T = \prod_{l \in S \setminus T} (\id - l^{-1} V_l F_l) : E_S (R) \longrightarrow E_S (R)
\]
converges pointwise. The projection $E_S (R) \to E_T (R)$ maps the image of $\varepsilon_T$ isomorphically to $E_T (R)$.\\
(2) If $R$ is unital let $e_T \in E_S (R)$ be the idempotent element
\[
e_T = \varepsilon_T (1) = \prod_{l \in S \setminus T} (1 - l^{-1} V_l (1)) \; .
\]
The projection induces an isomorphism of unital rings $e_T E_S (R) \silo E_T (R)$ mapping $e_T$ to $1 \in E_T (R)$. In particular, ring theoretically $E_T (R)$ is a direct factor of $E_S (R)$.
\end{proposition}

\begin{rem}
A description of Artin--Hasse idempotents as in (2) appears already in \cite{Ch} 1.1.6.
\end{rem}

\begin{proof}
Part (2) follows from (1). Using Corollary \ref{t29-n23} we obtain the commuting idempotents $l^{-1} V_l F_l$ and $\varepsilon_l = \id - l^{-1} V_l F_l$ in $\End (E_S (R))$. Choose a presentation \eqref{eq:4} such that $A$ has no $S$-torsion and such that $l$-multiplication on $A$ is invertible for all $l \in S \setminus T$. The kernel $I$ then has the same properties. For example we can take the standard presentation \eqref{eq:i1a} and tensor it with $\Lambda = \Z [1 / l$ for $l \in S \setminus T]$. Using Corollaries \ref{t29-n23} and \ref{t6} we obtain idempotent endomorphisms $\varepsilon_l$ as above on $E_S (A) = X_S (A)$ and $E_S (I) = X_S (I)$. They are easy to calculate explicitly. Consider the composition
\[
l^{-1} V_l F_l : A^S \longrightarrow A^{S/l} \longrightarrow A^S \; .
\]
We have
\[
l^{-1} V_l F_l (\bfx)_n = x_n \quad \text{if $l \mid n$ and $= 0$ if $l \nmid n$.}
\]
The resulting idempotent endomorphism of $A^S$
\begin{equation}
\label{eq:22}
\varepsilon_T = \prod_{l \in S \ohne T} (\id - l^{-1} V_l F_l)
\end{equation}
is explicitly given by the formula
\begin{equation}
\label{eq:23}
\varepsilon_T (\bfx)_n = x_n \quad \text{if $n \in T$, and $\varepsilon_T (\bfx)_n = 0$ if $n \in S \ohne T$.}
\end{equation}
Hence $\varepsilon_T$ is even a ring endomorphism of $A^S$ and we have seen that it leaves the subring $X_S (A)$ of $A^S$ invariant. It follows that the restriction of the projection $X_S (A) \to X_T (A)$ to $\varepsilon_T (X_S (A))$ induces an isomorphism of rings
\begin{equation}
\label{eq:24}
\varepsilon_T (X_S (A)) \xrightarrow{\sim} X_T (A) \; .
\end{equation}

By Corollary \ref{j7} c) we have $E_S (R) = X_S (A) / X_S (I)$ and similarly for $E_T (R)$. The isomorphisms \eqref{eq:24} for $A$ and similarly for $I$ induce an isomorphism
\[
\frac{\varepsilon_T (X_S (A))}{\varepsilon_T (X_S (I))} \xrightarrow{\sim} E_T (R) \; .
\]
We claim that the group on the left is the image of the map $\varepsilon_T : E_S (R) \to E_S (R)$. For this we need to check that
\[
\varepsilon_T (X_S (I)) = \varepsilon_T (X_S (A)) \cap X_S (I) \; .
\]
The inclusion "$\subset$" being clear, let $x = \varepsilon_T (y) \in X_S (I)$. Then $x = \varepsilon_T (x) \in \varepsilon_T (X_S (I))$ as required.
\end{proof}

We end this section by explaining a natural way to extend the $E_S$-theory from commutative to non-commutative rings. This will not be used in the sequel.

For a non-commutative ring $A$ the map $\langle \rangle : A \to A^S$ is no longer multiplicative. We define $X_S (A) \subset A^S$ to be the closed additive subgroup generated by all elements of the form $V_n (\langle a_1 \rangle \cdots \langle a_{\nu} \rangle )$ where $n \in S , \nu \ge 1$ and $a_1 , \ldots , a_{\nu} \in A$. For $\mu \ge 1, m \in S$ and $b_1 , \ldots , b_{\mu} \in A$ the relation (implied by Proposition \ref{t1})
\[
V_n (\langle a_1 \rangle \cdots \langle a_{\nu} \rangle) V_m (\langle b_1 \rangle \cdots \langle b_{\mu} \rangle) = (n,m) V_{[n,m]} (\langle a^{m'}_1 \rangle \cdots \langle a^{m'}_{\nu}\rangle \langle b^{n'}_1 \rangle \cdots \langle b^{n'}_{\mu} \rangle)
\]
shows that $X_S (A) \subset A^S$ is a closed topological subring of $A^S$. Here $n' = n / (n,m)$ and $m' = m / (n,m)$. For an ideal $I$ in $A$ we define $X_S (I,A) \subset A^S$ to be the closed subgroup generated by elements $V_n (\langle a_1 \rangle \cdots \langle a_{\nu} \rangle)$ as above where now $a_i \in I$ for at least one $1 \le i \le \nu$. Then $X_S (I , A)$ is a closed two-sided ideal in $X_S (A)$. For $n \in S$ we have by definition: 
\[
V_n (X_{S / n} (A)) \subset X_S (A) \quad \text{and} \quad V_n (X_{S /n} (I, A)) \subset X_S (I, A)  \; .
\]
The formula
\[
F_m V_n (\langle a_1 \rangle \cdots \langle a_{\nu} \rangle) = (n,m) V_{n'} (\langle a^{m'}_1 \rangle \cdots \langle a^{m'}_{\nu} \rangle)
\]
which follows from Proposition \ref{t1} shows that for $m \in S$ we have:
\[
F_m (X_S (A)) \subset X_{S/m} (A) \quad \text{and} \quad F_m (X_S (I,A)) \subset X_{S/m} (I,A) \; .
\]
For commutative $A$ we recover the previously defined ring $X_S (A)$ and ideal $X_S (I)$. For a possibly non-commutative ring $R$ consider the exact sequence $0 \to I \to \Z R \to R \to 0$ and set
\[
E_S (R) := X_S (\Z R) / X_S (I , \Z R) \; .
\]
This is a possibly non-commutative ring and there are ring-homomorphisms $F_m : E_S (R) \to E_{S/m} (R)$ for $m \in S$ and additive homomorphisms $V_n : E_{S/ n} (R) \to E_S (R)$ for $n \in S$. The analogues of the relations in Proposition \ref{t1} hold for the maps $F_m , V_n$ between the rings $E_S (R)$. Moreover there is a set-theoretical splitting
\[
\langle \; \rangle : R \to E_S (R) \, , \; r \mapsto \langle [r] \rangle \mod X_S (I , \Z R)
\]
of the natural projection $E_S (R) \to A /I = R$. Note that for $m \in S$ we have $F_m ( \langle r \rangle) = \langle r \rangle^m$ for the ring homomorphism $F_m$. The construction of $E_S (R)$ is functorial in $R$ and extends the previously defined theory from commutative to non-commutative rings. In \cite{H1} Hesselholt introduced an additive Witt group for non-commutative rings. It would be interesting to understand the relation between the two approaches.

\section{A relative de~Rham Witt complex} \label{sec:3}
In this section we define a relative de~Rham Witt complex $E_S \Omega^{\bfdot}_{R / R_0}$ for every $R_0$-algebra $R$ and every truncation set $S$. Here both $R$ and the base ring $R_0$ as well as their presentations are supposed to be commutative and unital. We also show that $E_S \Omega^{\bfdot}_{R / R_0}$ is isomorphic to Chatzistamatiou's relative de~Rham Witt complex \cite{Ch}.

Recall that for an $R$-module $M$ the exterior algebra $\Lambda^{\bfdot} M$ over $R$ is the graded commutative $R$-algebra obtained by dividing the tensor algebra of $M$ over $R$ by the two-sided ideal generated by the elements $m \otimes m$ in degree two where $m \in M$. If $M$ is a free $R$-module the algebra $\Lambda^{\bfdot} M$ is a free $R$-module as well. Let $\Omega^1_{R / R_0}$ be the $R$-module of relative differential $1$-forms of $R$ over $R_0$ and set $\Omega^{\bfdot}_{R / R_0} = \Lambda^{\bfdot} \Omega^1_{R / R_0}$. There is a unique extension of the $R_0$-linear differential $d : R \to \Omega^1_{R / R_0}$ to an $R_0$-linear endomorphism $d$ of $\Omega^{\bfdot}_{R / R_0}$ that makes $\Omega^{\bfdot}_{R / R_0}$ a differential graded $R_0$-algebra (or $R_0$-dga for short).

For convenience we make the following definition:

\begin{definition}
\label{tn41}
A presentation $(\pi)$ of an $R_0$-algebra $R$ is a commutative diagram with exact rows:
\begin{equation}
\label{eq:n41}
\xymatrix{
0 \ar[r] & I_0 \ar[r] \ar@{^{(}->}[d] & A_0 \ar[r]^{\pi_0} \ar@{^{(}->}[d] & R_0 \ar[r] \ar[d] & 0 \\
0 \ar[r] & I \ar[r] & A \ar[r]^{\pi} & R \ar[r] & 0
}
\end{equation}
where $A_0$ has no $\Z$-torsion and $A$ is a polynomial ring over $A_0$ via the vertical inclusion in the middle. We call the presentation $(\pi)$ free if in addition $A_0$ is a polynomial ring over $\Z$.
\end{definition}

Note that we may view $\pi$ as a map of $A_0$-algebras. By Theorem \ref{jc} the topological rings
\[
E_S (R_0,\pi_0) = X_S (A_0) / X_S (I_0) \quad \text{and} \quad E_S (R,\pi) = X_S (A) / X_S (I)
\]
are canonically isomorphic to $E_S (R_0)$ and $E_S (R)$ respectively. Consider the algebras of relative differential forms
\[
\Omega = \Omega^{\bfdot}_{A / A_0} \quad \text{and} \quad \Omega_{\Q} = \Omega^{\bfdot}_{A / A_0} \otimes_{\Z} \Q \; .
\]
Since $\Omega$ is a free $A_0$-module, $\Omega$ is a subring of $\Omega_{\Q}$. We view $\Omega^S_{\Q}$ as a differential graded ring with the differential $\dl$ defined by
\begin{equation} \label{eq:n421a}
\dl ((\omega_n)_{n \in S}) = (n^{-1} d\omega_n)_{n \in S} \; .
\end{equation}
The families $(\Omega^S)$ and $(\Omega^S_{\Q})$ are equipped with Frobenius and Verschiebung operators satisfying the relations in Proposition \ref{t1}. Since $\Omega^0 = A$, we have inclusions
\[
X_S (A) \subset A^S \subset \Omega^S \subset \Omega^S_{\Q} \; .
\]

\begin{definition}
\label{tn42}
We denote by $X^{\bfdot}_S (A)$ the closed differential graded subring of $(\Omega^S_{\Q}, \dl)$ generated by $X_S (A)$. Moreover $X^{\bfdot}_S (I , A)$ denotes the closed differential graded ideal in $X^{\bfdot}_S (A)$ generated by $X_S (I)$. Set
\[
E_S \Omega^{\bfdot} (\pi) = X^{\bfdot}_S (A) / X^{\bfdot}_S (I,A) \; .
\]
It is a Hausdorff topological dg-ring.
\end{definition}

By construction $E_S \Omega^0 (\pi) = E_S (R,\pi) \equiv E_S (R)$ and the differential of $E_S \Omega^{\bfdot} (\pi)$ is $E_S (R_0,\pi_0) \equiv E_S (R_0)$-linear, viewing $E_S (R)$ as an $E_S (R_0)$-algebra.

\begin{proposition}
\label{tn43}
For finite $S$ there is a unique homomorphism of differential graded $E_S (R_0)$-algebras
\begin{equation}
\label{eq:n42}
\alpha : \Omega^{\bfdot}_{E_S (R) / E_S (R_0)} \longrightarrow E_S \Omega^{\bfdot} (\pi)
\end{equation}
which is the identification $E_S (R) \equiv E_S (R,\pi)$ in degree zero. The map $\alpha$ is surjective.
\end{proposition}

\begin{proof}
Existence and uniqueness of $\alpha$ follow from the universal property of $\Omega^{\bfdot}$. By definition, $E_S \Omega^{\bfdot} (\pi)$ is generated as a dga by elements in degree zero. Hence $\alpha$ is surjective.
\end{proof}

We now discuss functoriality and uniqueness. Consider a commutative square
\begin{equation}
\label{eq:n43}
\xymatrix{
R_0 \ar[r] \ar[d] & R'_0 \ar[d]\\
R \ar[r] & R'
}
\end{equation}
and assume that we have a presentation $(\pi')$ of the $R'_0$-algebra $R'$
\begin{equation}
\label{eq:n44}
\xymatrix{
0 \ar[r] & I'_0 \ar[r] \ar@{^{(}->}[d] & A'_0 \ar[r]^{\pi'_0} \ar@{^{(}->}[d] & R'_0 \ar[r] \ar[d] & 0 \\
0 \ar[r] & I' \ar[r]                   & A' \ar[r]^{\pi'}                     & R' \ar[r]   & 0
}
\end{equation}
If $(\pi)$ is free there are a $\Z$-algebra homomorphism $A_0 \to A'_0$ and an $A_0$-algebra homomorphism $A \to A'$ which together with the horizontal maps in \eqref{eq:n43} induce a map from diagram \eqref{eq:n41} to diagram \eqref{eq:n44}. We obtain a commutative diagram
\begin{equation}
\label{eq:n45}
\xymatrix@C=0em{
X^{\bfdot}_S (I,A) \ar[d] \subset & X^{\bfdot}_S (A) \ar[d] \subset & (\Omega^{\bfdot}_{A / A_0} \otimes \Q)^S \ar[d] \\
X^{\bfdot}_S (I',A') \subset & X^{\bfdot}_S (A') \subset & (\Omega^{\bfdot}_{A' / A'_0} \otimes \Q)^S
}
\end{equation}
We will view $E_S \Omega^{\bfdot} (\pi')$ as a dga over $E_S (R_0)$ via the map $E_S (R_0) \to E_S (R'_0)$.

\begin{cor}
\label{tn44}
a) Assume we are given a square \eqref{eq:n43} and presentations $(\pi), (\pi')$ where $(\pi)$ is free. Then there is a unique continuous homomorphism of differential graded $E_S (R_0)$-algebras
\begin{equation}
\label{eq:n46}
E_S \Omega^{\bfdot} (\pi) \longrightarrow E_S \Omega^{\bfdot} (\pi')
\end{equation}
which is the map $E_S (R) \to E_S (R')$ in degree zero. \\
b) For two free resolutions $(\pi)$ and $(\pi')$ of the same $R_0$-algebra $R$ there is a unique topological isomorphism of differential graded $E_S (R_0)$-algebras \eqref{eq:n46} which is the identity on $E_S (R)$ in degree zero (more precisely it is the identification $E_S (R,\pi) \equiv E_S (R',\pi')$).
\end{cor}

\begin{proof}
Diagram \eqref{eq:n45} implies that the map \eqref{eq:n46} exists. It is unique because it is prescribed in degree zero and $E_S \Omega^{\bfdot} (\pi)$ is topologically dg-generated by elements in degree zero. The second assertion is a formal consequence of the first one.
\end{proof}

\begin{definition}
\label{tn45}
For a (commutative, unital) $R_0$-algebra $R$ we write $E_S \Omega^{\bfdot}_{R/R_0}$ for the $E_S (R_0)$-dga $E_S \Omega^{\bfdot} (\pi)$ where $(\pi)$ can be any free presentation as in \eqref{eq:n41}.
\end{definition}

This makes sense since by the corollary $E_S \Omega^{\bfdot} (\pi)$ is uniquely determined up to a unique topological isomorphism which is the identity in degree zero. It is immediate from the definitions that for $S = \{1 \}$ we have $E_S \Omega^{\bfdot} (\pi) = \Omega^{\bfdot}_{R / R_0}$.

Of course we could also have chosen a canonical free presentation to define $E_S \Omega^{\bfdot}_{R / R_0}$. For example we could have taken the following diagram with the obvious maps
\begin{equation}
\label{eq:n47}
\xymatrix{
0 \ar[r] & I_0 \ar[r] \ar@{^{(}->}[d] & \Z [R_0] \ar[r] \ar@{^{(}->}[d] & R_0 \ar[r] \ar[d] & 0 \\
0 \ar[r] & I \ar[r] & \Z [R_0] [R] \ar[r] & R \ar[r] & 0
}
\end{equation}
In the introduction, where $R_0$ was $\Z$ we used the free presentation
\[
\xymatrix{
0 \ar[r] & 0 \ar[r] \ar[d] & \Z \ar[r] \ar@{^{(}->}[d] & \Z \ar[r] \ar[d] & 0 \\
0 \ar[r] & I \ar[r] & \Z [R] \ar[r] & R \ar[r] & 0
}
\]
The extra flexibility in the choice of resolutions is very useful though. For example we have:

\begin{cor}
\label{tn46}
Assume that $R$ is a finitely generated $R_0$-algebra and let $d$ be the minimal number of generators. Then $E_S \Omega^i_{R / R_0} = 0$ for $i > d$.
\end{cor}

\begin{proof}
We can choose a free presentation \eqref{eq:n41} where $A = A_0 [t_1 , \ldots , t_d]$. Then $\Omega^i_{A / A_0} = 0$ for $i > d$ and the claim follows by the construction of $E_S \Omega^{\bfdot}_{R / R_0}$.
\end{proof}

Our next goal is to equip the family $(E_S \Omega^{\bfdot}_{R / R_0})$ for varying $S$ with Frobenius and Verschiebung maps and to determine the relations between $V_n , F_m, \dl$ and $\langle \rangle_S$. This is based on the following lemma whose proof consists of simple verifications. Consider a polynomial algebra $A$ over a commutative unital base ring $A_0$ without $\Z$-torsion. Let $(\Omega^S_{\Q} , \dl)$ be the dga defined in \eqref{eq:n421a} where $\Omega = \Omega^{\bfdot}_{A / A_0}$. The family $(\Omega^S_{\Q})$ comes with continuous maps $V_n , F_m$ satisfying the relations of Proposition \ref{t1}.

\begin{lemma}
\label{tn47}
For $n,m \in S$, setting $n' = n / (n,m)$ and $m' = m / (n,m)$ the following relations hold:\\
(1) $V_n \dl = n \, \dl V_n : \Omega^{S/n}_{\Q} \longrightarrow \Omega^S_{\Q}$\\
(2) $mF_m \dl = \dl F_m : \Omega^S_{\Q} \longrightarrow \Omega^{S/m}_{\Q}$\\
(3) $F_n \dl V_n = \dl : \Omega^{S/n}_{\Q} \longrightarrow \Omega^{S/n}_{\Q}$\\
(4) $m' F_m \dl V_n = \dl V_{n'} F_{m'} : \Omega^{S/n}_{\Q} \longrightarrow \Omega^{S/m}_{\Q}$\\
(5) $n' F_m \dl V_n = V_{n'} F_{m'} \dl : \Omega^{S/n}_{\Q} \longrightarrow \Omega^{S/m}_{\Q}$\\
(6) Choose $i,j \in \Z$ with $im'+ jn' = 1$. Then
\[
F_m \dl V_n = i\, \dl V_{n'} F_{m'} + j V_{n'} F_{m'} \dl : \Omega^{S/n}_{\Q} \longrightarrow \Omega^{S/m}_{\Q} \; .
\]
(7) $F_n \dl \langle a \rangle_S = \langle a^{n-1} \rangle_{S/n} \dl \langle a \rangle_{S/n}$ for $a \in A$ and $\langle a \rangle_S = (a^{\nu})_{\nu \in S} \in A^S \subset \Omega^S_{\Q}$\\
(8) $F_m \dl V_n \langle a \rangle_{S/n} = i\, \dl V_{n'} \langle a^{m'} \rangle_{S / nm'} + j V_{n'} \langle a^{m'-1} \rangle_{S /n'm} \dl \langle a \rangle_{S/n'm}$ for $a \in A$\\
(9) $V_n (\omega_0 \dl \omega_1 \cdots \dl \omega_r) = V_n (\omega_0) \dl V_n (\omega_1) \cdots \dl V_n (\omega_r)$ for $\omega_0 , \ldots , \omega_r \in \Omega^{S/n}_{\Q}$.
\end{lemma}

\begin{proof}
(1), (2) follow by calculating both sides explicitly.\\
(3) Proposition \ref{t1} and (1) give
\[
n \dl = F_n V_n \dl \overset{\text{(1)}}{=} n F_n \dl V_n \; .
\]
(4) $m F_m \dl V_n \overset{\text{(2)}}{=} \dl F_m V_n = (n,m) \dl V_{n'} F_{m'}$ by Proposition \ref{t1}.\\
(5) $n F_m \dl V_n \overset{\text{(1)}}{=} F_m V_n \dl = (n,m) V_{n'} F_{m'} \dl$ by Proposition \ref{t1}.\\
(6) follows from (4) and (5).\\
(7) $F_n \dl \langle a \rangle = F_n \dl (a^{\nu}) = F_n (a^{\nu-1} da) = (a^{\nu n-1} da)$. On the other hand
\[
\langle a^{n-1} \rangle \dl \langle a \rangle = ((a^{n-1})^{\nu}) (a^{\nu-1} da) = (a^{\nu n-1} da) \; .
\]
(8) follows from (6) and (7).\\
(9) Easy calculation. Even easier for $r = 1$, the general case following by induction.
\end{proof}

\begin{theorem}
\label{tn48}
For any presentation \eqref{eq:n41} of an $R_0$-algebra $R$ we have inclusions not only in $\Omega^S_{\Q}$ but in fact in $\Omega^S$
\[
X^{\bfdot}_S (I,A) \subset X^{\bfdot}_S (A) \subset \Omega^S \quad \text{where} \; \Omega = \Omega^{\bfdot}_{A / A_0} \; .
\]
For $n,m \in S$ the Frobenius and Verschiebung maps $F_m$ and $V_n$ for the family $(\Omega^S)$ respect these inclusions.
\end{theorem}

\begin{proof}
As a topological (additive) subgroup of $\Omega^S_{\Q}$ the ring $X^{\bfdot}_S (A)$ is generated by elements of the form
\begin{equation}
\label{eq:n435}
\omega = V_{n_0} \langle a_0 \rangle \dl V_{n_1} \langle a_1 \rangle \cdots \dl V_{n_r} \langle a_r \rangle \quad \text{with} \; r \ge 0 , n_i \in S , a_i \in A \; .
\end{equation}
Similarly $X^{\bfdot}_S (I,A)$ is generated by $\omega$'s with at least one $a_i \in I$. The formula
\begin{equation}
\label{eq:n436}
\dl V_n \langle a \rangle = (\delta_{n \mid \nu} a^{\frac{\nu}{n} - 1} da)_{\nu \in S} \in \Omega^S \quad \text{with} \; a \in A
\end{equation}
now implies that $X^{\bfdot}_S (A) \subset \Omega^S$. Recall that $\delta_{n \mid \nu} = 1$ if $n \mid \nu$ and $\delta_{n \mid \nu} = 0$ if $n \nmid \nu$. The continuous Frobenius ring homomorphism $F_m : \Omega^S \to \Omega^{S/ m}$ maps $X_S (A)$ to $X_{S/m} (A)$. Moreover formula (8) of Lemma \ref{tn47} shows that $F_m \dl V_n \langle a \rangle \in X^{\bfdot} (A)$. Hence we have $F_m (X^{\bfdot}_S (A)) \subset X^{\bfdot}_{S/m} (A)$. The same argument implies that $F_m (X^{\bfdot}_S (I,A)) \subset X^{\bfdot}_{S / m} (I,A)$. The corresponding assertions for $V_n$ follow from formula (9) of Lemma \ref{tn47} since $V_n \verk V_m = V_{nm}$.
\end{proof}

From the theorem we get induced Frobenius and Verschiebung maps between the groups $E_S \Omega^{\bfdot} (\pi)$ where $(\pi)$ is any presentation of the $R_0$-algebra $R$. The maps $E_S \Omega^{\bfdot} (\pi) \to E_S \Omega^{\bfdot} (\pi')$ in Corollary \ref{tn44} are compatible with $F_m$ and $V_n$. This follows because they are induced by diagram \eqref{eq:n45}. Thus ``the'' dgas $E_S \Omega^{\bfdot}_{R / R_0}$ of Definition \ref{tn45} come equipped with Frobenius and Verschiebung maps. The next corollary is clear from the definitions.

\begin{cor}
\label{tn49}
The ring homomorphisms $F_m : E_S \Omega^{\bfdot}_{R / R_0} \to E_{S/m} \Omega^{\bfdot}_{R / R_0}$ for $m \in S$ and the additive maps $V_n : E_{S/n} \Omega^{\bfdot}_{R / R_0}\to E_S \Omega^{\bfdot}_{R / R_0}$ for $n \in S$ satisfy all the relations corresponding to those in Proposition \ref{t1} and Lemma \ref{tn47}.
\end{cor}

There is a ``ghost'' map on $E_S \Omega^{\bfdot}_{R / R_0}$. Formulas \eqref{eq:n435} and \eqref{eq:n436} show that the image of $X^{\bfdot}_S (I,A)$ in $(\Omega^{\bfdot }_{A / A_0})^S$ has components in the differential graded ideal $J^{\bfdot}$ of $\Omega^{\bfdot}_{A /A_0}$ generated by $I$. Since $\Omega^{\bfdot}_{A / A_0} / J^{\bfdot} = \Omega^{\bfdot}_{R / R_0}$ we therefore obtain a homomorphism of graded rings
\begin{equation}
\label{eq:n437}
\Gh^{\bfdot}_S : E_S \Omega^{\bfdot}_{R / R_0} \longrightarrow (\Omega^{\bfdot}_{R / R_0})^S \; .
\end{equation}
Its degree zero component $\Gh^0_S = \Gh_S$ is the ghost map \eqref{eq:18}. By construction $\Gh^{\bfdot}_S$ is compatible with Frobenius and Verschiebung and we have
\begin{equation}
\label{eq:n438}
d \Gh^{\bfdot}_S (\omega) = (\nu)_{\nu \in S} \Gh^{\bfdot}_S (\dl \omega) \; .
\end{equation}
Here the differential on the left is taken componentwise.

The functor $S \mapsto E_S \Omega^{\bfdot}_{R / R_0}$ is continuous in the following sense:

\begin{proposition} \label{t410}
The natural map
\[
E_S \Omega^{\bfdot}_{R / R_0} \longrightarrow \varprojlim_{T \subset S} E_T \Omega^{\bfdot}_{R / R_0}
\]
is a topological isomorphism of $E_S (R_0) = \varprojlim_{T \subset S} E_T (R_0)$ dga's. Here $T$ runs over the directed poset of finite truncation subsets of $S$.
\end{proposition}

\begin{proof}
Choose a free resolution $(\pi)$ as in \eqref{eq:n41}. For finite $T \subset S$ the elements of $X^{\bfdot}_T (A)$ are finite $\Z$-linear combinations of forms $\omega$ as in \eqref{eq:n435}. It follows that the projection $X^{\bfdot}_S (A) \to X^{\bfdot}_T (A)$ is surjective. Similarly we see that $X^{\bfdot}_S (I,A) \to X^{\bfdot}_T (I,A)$ is surjective as well. The argument in the proof of Proposition \ref{tx2}, 3) shows that topologically
\[
X^{\bfdot}_S (A) = \varprojlim X^{\bfdot}_T (A) \quad \text{and} \quad X^{\bfdot}_S (I,A) = \varprojlim X^{\bfdot}_T (I,A) \; .
\]
Applying Lemma \ref{t23} to the exact sequences:
\[
0 \longrightarrow X^{\bfdot}_T (I,A) \longrightarrow X^{\bfdot}_T (A) \longrightarrow E_T \Omega^{\bfdot}_{R / R_0} \longrightarrow 0
\]
the assertion follows.
\end{proof}

For the comparison of $E_S \Omega^{\bfdot}_{R / R_0}$ with the relative de~Rham Witt complexes in \cite{Ch}, \cite{LZ} and \cite{I}, we need some information about $X^{\bfdot}_S (A)$ where $A$ is a polynomial algebra over a ring $A_0$ without $\Z$-torsion and $S$ is finite. Since $X^{\bfdot}_S (A)$ is an $X_S (A_0)$-dga with $X^0_S (A) = X_S (A)$, there is a unique map of $X_S (A_0)$-dga's which is the identity in degree zero:
\[
\alpha : \Omega^{\bfdot}_{X_S (A) / X_S (A_0)} \longrightarrow X^{\bfdot}_S (A) \; .
\]
The map $\alpha$ is surjective since $X^{\bfdot}_S (A)$ is dg-generated by elements in degree zero. By construction, $X^{\bfdot}_S (A)$ has no $\Z$-torsion, so that we get an induced surjective map
\[
\oalpha : \Omega^{\bfdot}_{X_S (A) / X_S (A_0)} / \Z\text{-torsion} \twoheadrightarrow X^{\bfdot}_S (A) \; .
\]

\begin{lemma}
\label{t411}
For finite $S$ the map $\oalpha$ is an isomorphism of $X_S (A_0)$-dga's.
\end{lemma}

\begin{proof}
By Lemma \ref{tn13}, for a ring $B$ the inclusion $X_S (B) \subset B^S$ becomes an isomorphism after tensoring with $\Q$. In particular, we have
\[
X_S (B) \otimes \Q = B^S \otimes \Q = (B \otimes \Q)^S\quad \text{for} \; B = A_0 , A \; .
\]
Thus the natural map
\[
\Omega^{\bfdot}_{X_S (A) / X_S (A_0)} \longrightarrow \Omega^{\bfdot}_{A^S / A^S_0} = (\Omega^{\bfdot}_{A / A_0})^S
\]
becomes an isomorphism after tensoring with $\Q$. In particular
\[
\Omega^{\bfdot}_{X_S (A) / X_S (A_0)} / \Z\text{-torsion} \hookrightarrow (\Omega^{\bfdot}_{A / A_0})^S
\]
is injective, noting that $\Omega^{\bfdot}_{A / A_0}$ has no $\Z$-torsion. We get a commutative diagram of dg-rings
\[
\xymatrix@C=0em{
\Omega^{\bfdot}_{X_S (A) / X_S (A_0)} / \Z\text{-torsion} \ar@{^{(}->}[rrr] \ar@{>>}[d]_{\bar{\alpha}} & & & ((\Omega^{\bfdot }_{A / A_0})^S, d^S) \subset & ((\Omega^{\bfdot }_{A / A_0})^S \otimes \Q , d^S) \ar[d]^{\wr \varphi} \\
X^{\bfdot}_S (A) \ar@{^{(}->}[rrrr] & & & & ((\Omega^{\bfdot }_{A / A_0})^S \otimes \Q , \dl) \; .
}
\]
Here $\varphi$ is the isomorphism of dg-rings given by the formula
\[
\varphi ((\omega_{\nu})_{\nu \in S}) = (\nu^{-\deg \omega_{\nu}} \omega_{\nu})_{\nu \in S}
\]
in case all $\omega_{\nu}$ are homogeneous forms. The commutativity of the diagram needs to be checked in degree zero only where it is trivial. It follows that $\oalpha$ is also injective and hence an isomorphism.
\end{proof}

\begin{rem}
The above diagram shows that for a polynomial algebra $A$ over a $\Z$-torsion free ring $A_0$ the dga $X^ {\bfdot}_S (A)$ is isomorphic as an $X_S (A_0)$-dga to the image of $\Omega^{\bfdot}_{X_S (A) / X_S (A_0)}$ in $(\Omega^{\bfdot}_{A^S / A^S_0} , d) = (\Omega^S_{A / A_0} , d^S)$ under the ghost map $X_S (A) \subset A^S$. In this description, Frobenius $\Fh_m$ and Verschiebung $\Vh_n$ are given by the restrictions of $\Fh_m = \varphi^{-1} \verk F_m \verk \varphi$ and $\Vh_n = \varphi^{-1} \verk V_n \verk \varphi$. Explicitly we have the formulas
\[
\Fh_m : \Omega^{\bfdot S}_{A / A_0} \otimes \Q \longrightarrow \Omega^{\bfdot S/m}_{A / A_0} \otimes \Q \; \text{with} \; \Fh_m ((\omega_{\nu})_{\nu \in S}) = (n^{-\deg \omega_{\nu n}} \omega_{\nu n})_{\nu \in S / n}
\]
and 
\[
\Vh_n : \Omega^{\bfdot S/n}_{A / A_0} \otimes \Q \longrightarrow \Omega^S_{A / A_0} \otimes \Q \; \text{with} \; \Vh_n ((\omega_{\nu})_{\nu \in S/n}) = n (n^{\deg \omega_{\nu / n}} \delta_{n \mid \nu} \omega_{\nu / n})_{\nu \in S} \; .
\]
Here the components $\omega_{\nu}$ are supposed to be homogenous. We learned from Jim Borger that this point of view is known to experts. Geometrically our approach via free presentations corresponds to using closed embeddings into smooth schemes. As we saw, for the definition of $X^{\bfdot} (A)$ one either needs to introduce denominators in the definition of Frobenius or in the definition of the differential. In the latter case the formulas are simpler which is the reason why we chose that option. We refer to section 2 of the arXiv v2 version of \cite{Ch} where a related description of the de~Rham Witt complex for $\lambda$-rings is given. 
\end{rem}

For an $R_0$-algebra $R$, the functor $S \mapsto E_S \Omega^{\bfdot}_{R / R_0}$ is a Witt complex over $R$ with $W (R_0)$-linear differential in the sense of \cite{H} \S\,4 and \cite{Ch} section 1.2. This follows from Corollary \ref{tn49} and Proposition \ref{t410}. The relative de~Rham Witt complex $S \mapsto W_S \Omega^{\bfdot}_{R / R_0}$ defined in \cite{Ch} section 1.2 is an initial object in this category. Hence there is a unique morphism $\lambda = (\lambda_S)$ of Witt complexes with
\[
\lambda_S : W_S \Omega^{\bfdot}_{R / R_0} \longrightarrow E_S \Omega^{\bfdot}_{R / R_0} \; .
\]
In degree zero, $\lambda_S$ is given by the identification $W_S (R) \equiv E_S (R)$. 

\begin{theorem}
\label{t3012}
For every $R_0$-algebra $R$ the morphism $\lambda = (\lambda_S)$ of Witt complexes is an isomorphism.
\end{theorem}

\begin{proof}
We may assume that $S$ is finite. It suffices to construct maps of differential graded rings
\[
\mu_S : E_S \Omega^{\bfdot}_{R / R_0} \longrightarrow W_S \Omega^{\bfdot}_{R / R_0}
\]
which are the identifications $E_S (R) \equiv W_S (R)$ in degree zero. Then both $\mu_S \verk \lambda_S$ and $\lambda_S \verk \mu_S$ are self maps of dg-rings with degree zero components being the identities. Since both $W_S \Omega^{\bfdot}_{R / R_0}$ and $E_S \Omega^{\bfdot}_{R / R_0}$ are generated as differential rings by degree zero elements it will follow that $\mu_S \verk \lambda_S = \id$ and $\lambda_S \verk \mu_S = \id$. Choose a presentation \eqref{eq:n41} of the $R_0$-algebra $R$ and consider the composition of canonical dg-ring maps
\[
\Omega^{\bfdot}_{X_S (A) / X_S (A_0)} \equiv \Omega^{\bfdot}_{W_S (A) / W_S (A_0)} \longrightarrow W_S \Omega^{\bfdot}_{A / A_0} \; .
\]
Here we have used Corollary \ref{t6} and the discussion following it to identify $X_S (B)$ with $W_S (B)$ for $B = A_0 , A$. Using Proposition 1.2.17 of \cite{Ch} which extends to (spectra of) polynomial algebras $A$ in arbitrary many variables by an inductive limit argument, it follows that $W_S \Omega^{\bfdot}_{A / A_0}$ has no $\Z$-torsion. Note that the proof of \cite{Ch} Proposition 1.2.17 rests on the work of Langer and Zink, c.f. \cite{LZ} Corollary 2.18. Using Lemma \ref{t411} we therefore obtain maps of dg-rings
\[
X^{\bfdot}_S (A) \overset{\oalpha^{-1}}{\silo} \Omega^{\bfdot}_{X_S (A) / X_S (A_0)} / \Z\text{-torsion} \longrightarrow W^{\bfdot}_S \Omega_{A / A_0} \longrightarrow W_S \Omega^{\bfdot}_{R / R_0} \; .
\]
Since the composition $X_S (I) \longrightarrow W_S (A) \to W_S (R)$ is zero, it follows that $X^{\bfdot}_S (I,A)$ is mapped to zero in $W_S \Omega^{\bfdot}_{R / R_0}$. Hence we get the desired homomorphism of dg-rings:
\[
\mu_S : E_S \Omega^{\bfdot}_{R / R_0} = X^{\bfdot}_S (A) / X^{\bfdot}_S (I,A) \longrightarrow W_S \Omega^{\bfdot}_{R / R_0} \; .
\]
\end{proof}

\begin{rem}
For a $\Z_{(p)}$-algebra resp. an $\F_p$-algebra $R_0$, Chatzistamatiou's complex \\
$W_n \Omega^{\bfdot}_{R / R_0} := W_{S_n} \Omega^{\bfdot}_{R / R_0}$ for $S_n = \{ 1 , \ldots , p^{n-1} \}$ reduces to the de~Rham Witt complexes of Langer, Zink \cite{LZ} resp. Bloch, Deligne, Illusie \cite{I}. See \cite{Ch} for the details. Thus $E_n \Omega^{\bfdot}_{R / R_0} := E_{S_n} \Omega^{\bfdot}_{R / R_0}$ provides simple descriptions of these complexes as well.
\end{rem}

It is interesting to make the differential $\dl$ explicit in the power series description of big Witt vectors. The relevant truncation set is $S = \N$ and we write $E \Omega^{\bfdot}$ for $E_S \Omega^{\bfdot}$ etc. For a presentation \eqref{eq:n41} write as before $\Omega = \Omega^{\bfdot}_{A / A_0}$ and $\Omega_{\Q} = \Omega^{\bfdot} \otimes \Q$. We may identify $\Omega^{\N}$ with $t\Omega \llbracket t \rrbracket$ equipped with the Hadamard product $\ast$ of coefficientwise multiplication. Similarly $\Omega^{\N}_{\Q} \equiv t \Omega_{\Q} \llbracket t \rrbracket$. Let $d_A$ denote the differential of $\Omega^{\bfdot}_{A / A_0}$. For $\omega = (\omega_{\nu})_{\nu \in \N} \in \Omega^{\N}_{\Q}$ we have by definition
\[
\dl \omega = (\nu^{-1} d_A \omega_{\nu})_{\nu \in \N} \in \Omega^{\N}_{\Q} \; .
\]
Identifying $\omega$ with a power series in $t \Omega_{\Q} \llbracket t \rrbracket$
\[
\omega = \sum^{\infty}_{\nu = 1} \omega_{\nu} t^{\nu}
\]
we may formally write
\[
\dl \omega = \int^t_0 d_A \omega \frac{dt}{t} \; .
\]
With our identifications, the power series in $X (A) \subset t\Omega \llbracket t \rrbracket$ are those of the form
\[
f_Q = t \partial_t \log Q \quad \text{for} \; Q \in 1 + tA \llbracket t \rrbracket \; .
\]
This follows from part b) of Proposition \ref{t2}. We have
\[
\dl f_Q = d_A \int^t_0 \partial_t \log Q \, dt = d_A \log Q = Q^{-1} d_A Q \quad \text{in} \;  t \Omega \llbracket t \rrbracket \; .
\]
In particular it follows again that $\dl X (A) \subset \Omega^{\N} = t \Omega \llbracket t \rrbracket$. Thus $X^{\bfdot} (A)$ is the closed subring of $(t \Omega \llbracket t \rrbracket , \ast)$ generated by elements of the form $f_P = t\, P^{-1} \partial_t P$ and $\dl f_Q = Q^{-1} d_A Q$ for $P,Q$ in $1 + tA \llbracket t \rrbracket$.

\begin{rem}
Unlike $\Omega^{\bfdot}_{\Z [R] / \Z}$ the ring $\Omega^{\bfdot}_{\Z R / \Z}$ can have $\Z$-torsion. Here is an example. Consider the $\F_p$-algebra $R = \F_p [u] / (u^p)$. The relation $[u]^p = 1$ in $\Z R$ implies that $p [u]^{p-1} d[u] = 0$. Multiplying with $[u]$ gives $p\, d [u] = 0$. On the other hand the image $du$ of $d [u]$ under the map $\Omega^1_{\Z R / \Z} \to \Omega^1_{R/\Z}$ is non-zero and hence $d [u]$ is a non-zero torsion element in $\Omega^1_{\Z R / \Z}$.
\end{rem}

We finish this section with some remarks on the overconvergent theory. Consider a finitely generated $R_0$-algebra $R$ with generators $r_1 , \ldots , r_d$. For $0 \neq r \in R$ define $\deg r$ to be the minimal degree among all polynomials $P \in R_0 [t_1 , \ldots , t_d]$ with $r = P (r_1 , \ldots , r_d)$. We set $\deg 0 = - \infty$. For two sequences $(a_n) , (b_n)$ we write $a_n \ll b_n$ for $n \to \infty$ if there is a constant $C > 0$ with $a_n \le C b_n$ for all sufficiently large $n$. Set
\[
W^{\dagger}_S (R) = \{ \bfr \in R^S \mid \deg r_n \ll n \log n \; \text{for} \; n \to \infty \; , \; n \in S \} \; .
\]
This subset of $W_S (R)$ is independent of the choice of generators used to define $\deg$. It was introduced in \cite{DLZ1} for $S = \{ 1 , p , p^2 , \ldots \}$ and base fields $R_0$ of characteristic $p$ and it was noted that $W^{\dagger}_S (R)$ is actually a subring of $W_S (R)$ respected by $F_p$ and $V_p$. In our approach we obtain $W^{\dagger}_S (R)$ as follows. Consider a presentation $(\pi)$ as in \eqref{eq:n41} where $A_0$ is only assumed to have no $S$-torsion and where $A$ is a finitely generated polynomial algebra over $A_0$ equipped with its usual degree map. Set
\[
X^{\dagger}_S (A) := \{ \bfx \in X_S (A) \mid \deg x_n \ll n \log n \; \text{for} \; n \to \infty \; , \; n \in S \}
\]
and $X^{\dagger}_S (I,A) := X_S (I) \cap X^{\dagger}_S (A)$. Then $X^{\dagger}_S (I, A)$ is an ideal in the ring $X^{\dagger}_S (A)$ and we define a subring of $E_S (R, \pi)$ by setting
\[
E^{\dagger}_S (R , \pi) := X^{\dagger}_S (A) / X^{\dagger}_S (I, A) \; .
\]
Note that in contrast to $E_S (R, \pi)$ the ring $E^{\dagger}_S (R, \pi)$ depends on the base ring $R_0$. Note that $E^{\dagger}_S (R, \pi)$ is an $E_S (R_0)$-algebra. In the natural sense it is invariant under $F_n , V_n$ for $n \in S$ and the Teichm\"uller map takes values in $E^{\dagger}_S (R, \pi)$. With these structures the rings $E^{\dagger}_S (R, \pi)$ for different presentations $\pi$ are isomorphic up to unique isomorphisms. Namely simple estimates show that the bijections
\[
\psi_S : W_S (R) = R^S \silo E_S (R ,\pi) \; , \; \psi_S (\bfr) = \sum_{n \in S } V_n \langle r_n \rangle
\]
restrict to bijections
\[
\psi^{\dagger}_S : W^{\dagger}_S (R) \silo E^{\dagger}_S (R , \pi) \; .
\]
If one views $\psi_S$ as a ring homomorphism, it follows that $W^{\dagger}_S (R)$ is a subring of $W_S (R)$ and $\psi^{\dagger}_S$ is a ring isomorphism. One can generalize the definition of the rings $E^{\dagger}_S (R)$ by considering more general growth functions $\tau : S \to (0,\infty]$ which satisfy the following conditions:\\
i) $\tau (n) \ge \varepsilon > 0$ for all $n \in S$\\
ii) $\tau (n) \le \tau (nm) \le c (m) \tau (n)$ for any $m \in S$ and all $n \in S / m$. \\
Here $c (m)$ is a constant depending on $m$. Then
\[
X^{\tau}_S (A) := \{ \bfx \in X_S (A) \mid \deg x_n \ll n \tau (n) \; \text{for} \; n \to \infty \; , \; n \in S \}
\]
is a subring of $X_S (A)$ and proceeding as before one obtains a subring $E^{\tau}_S (R)$ of $E_S (R)$ which is respected by $\langle \rangle_S$ and $F_n , V_n$ for $n \in S$. For $\tau \equiv \infty$ one obtains $E_S (R)$. The map $\psi_S$ restricts to a bijection of $E^{\tau}_S (R)$ with
\[
W^{\tau}_S (R) := \{ \bfr \in R^S \mid \deg r_n \ll n \tau (n) \; \text{for} \; n \to \infty \; , \; n \in S \} \; .
\]
The de~Rham Witt theory has a corresponding generalization. Consider the $dg$-ring $\Omega^{S, \tau}_{\Q}$ which consists of all elements $\uomega = (\omega_n)_{n \in S}$ in $\Omega^S_{\Q}$ where $\Omega_{\Q} = \Omega^{\bfdot}_{A / A_0} \otimes \Q$ such that $\deg \omega_n \ll n \tau (n)$. Here $A_0$ is supposed to be a polynomial algebra over $\Z$ and $A$ is a polynomial algebra over $A_0$. Set
\[
\deg \sum P_{i_1 \ldots i_r} dt_{i_1} \wedge \ldots \wedge dt_{i_r} := \max (\deg P_{i_1 \ldots i_r} + i_1 + \ldots + i_r) \; .
\]
Consider $X^{\bfdot \tau}_S (A)$, the $dg$-graded subalgebra of $\Omega^{S,\tau}_{\Q}$ generated by $X^{\tau}_S (A)$ and the $dg$-ideal $X^{\bfdot \tau}_S (I,A)$ in $X^{\bfdot \tau}_S (A)$ generated by $X^{\tau}_S (I,A)$. Both lie in $(\Omega^{\bfdot}_{A / A_0})^S \cap \Omega^{S , \tau}_{\Q}$ and
\[
E^{\tau}_S \Omega^{\bfdot} (\pi) := X^{\bfdot \tau}_S (A) / X^{\bfdot \tau}_S (I, A) 
\]
is an $E_S (R_0)$-dga with Frobenius and Verschiebung maps. As before there are unique isomorphisms between these $E_S (R_0)$-dga's for free presentations $(\pi)$ and $(\pi')$ that are the identifications $E^{\tau}_S (R, \pi) \equiv E^{\tau}_S (R , \pi')$ in degree zero. For $\tau \equiv \infty$ we get $E_S \Omega^{\bfdot} (\pi)$ back. It would be interesting to compare $E^{\tau}_S \Omega^{\bfdot} (\pi)$ for $\tau (n) = \log n$ and $S = \{ 1 , p , p^2 , \ldots \}$ to the overconvergent de~Rham Witt complex wich was introduced in \cite{DLZ2}.
\section{Frobenius lifts} \label{sec:4}
From this section on we return to rings $R$ that are commutative and possibly non-unital. We explain how to prove the well known lifting property of the Frobenius map and the Cartier--Dieudonn\'e lemma in our context and discuss the latter for $E_S (\Z R)$. The Frobenius lifting is proved for completeness only The Cartier--Dieudonn\'e lemma will be needed in section \ref{sec:6}.

For every $p \in S$ with $S = S/p$ the Frobenius endomorphism $F_p$ of $E_S (R)$ lifts the $p$-th power Frobenius endomorphism of the $\F_p$-algebra $E_S (R) / p E_S (R)$. In fact there is a more precise result.

\begin{proposition}\label{t9n}
For any $p \in S$ there is a unique family of maps $\delta_p : E_S (R) \to E_{S/p} (R)$ which is functorial in rings $R$ such that we have
\begin{equation}
\label{eq:11a}
F_p (c) = \tc^p + p \delta_p (c) \quad \text{for} \; c \in E_S(R) \; .
\end{equation}
Here $\tc$ is the image of $c$ under the projection $E_S (R) \to E_{S/p} (R)$. An explicit formula for $\delta_p (c)$ is given in \eqref{eq:11d} below. In particular
\[
F_p (c) \equiv \tc^p \mod pE_{S/p} (R) \; .
\]
\end{proposition}

\begin{proof}
The ring $E_{S/p} (\Z R) = X_{S/p} (\Z R) \subset (\Z R)^{S/p}$ has no $\Z$-torsion. Hence formula \eqref{eq:11a} determines $\delta_p$ uniquely on $E_S (\Z R)$. Since $E_S (R)$ is a quotient of $E_S (\Z R) = X_S (\Z R)$ the uniqueness assertion follows. By Theorem \ref{jc} and formula \eqref{j6}, every element $c$ of $E_S (R)$ is uniquely of the form
\[
c = \sum_{n \in S} V_n \langle r_n \rangle \quad \text{with} \; r_n \in R \; .
\]
We first define $\delta_p$ on the elements $V_n \langle r \rangle$ for $r \in R$. For $p \nmid n$ Proposition \ref{t1} and the relation $F_p \langle r \rangle = \langle r^p \rangle$ give the formula
\[
F_p V_n \langle r \rangle - (V_n \langle r \rangle )^p = (1 - n^{p-1}) F_pV_n \langle r \rangle \; .
\]
Hence we must define:
\begin{equation}
\label{eq:11b}
\delta_p (V_n \langle r \rangle ) = p^{-1} (1 - n^{p-1}) F_p V_n \langle r \rangle \; .
\end{equation}
For $p \mid n$ a short calculation, again using Proposition \ref{t1} shows that
\[
F_p V_n \langle r \rangle - (V_n \langle r \rangle)^p = p b - p^{p-1} V_p (b^p) \quad \text{where} \; b = V_{n/p} \langle r \rangle \; .
\]
Hence we have to set
\begin{equation}
\label{eq:11c}
\delta_p (V_n \langle r \rangle) = b - p^{p-2} V_p (b^p)\; .
\end{equation}
Let $\Delta$ be the diagonal in $S^p$. Then the cyclic permutation $\tau = (2 , \ldots , p , 1)$ acts without fixed points on $S^p \ohne \Delta$. Setting $\alpha_n = V_n \langle r_n \rangle$ and using \eqref{eq:11b} and \eqref{eq:11c} we have:
\begin{align*}
F_p \Big( \sum_{n \in S} \alpha_n \Big) - \Big( \sum_{n \in S} \alpha_n \Big)^p & = \sum_{n \in S} (F_p (\alpha_n) - \alpha^p_n) - \sum_{\bfn \in S^p \setminus \Delta} \alpha_{n_1} \cdots \alpha_{n_p} \\
& = p \sum_{n \in S} \delta_p (\alpha_n) - p \sum_{\bfn \in (S^p \ohne \Delta) / \tau} \alpha_{n_1} \cdots \alpha_{n_p} \; .
\end{align*}
Defining $\delta_p (c)$ by the following formula we therefore get a solution to equation \eqref{eq:11a} which is functorial in $R$
\begin{equation}
\label{eq:11d}
\delta_p (c) = \sum_{n\in S} \delta_p (\alpha_n) - \sum_{\bfn \in (S^p \ohne \Delta) / \tau} \alpha_{n_1} \cdots \alpha_{n_p} \; .
\end{equation}
\end{proof}

The following result determines the image of the ghost map $\Gh_S : E_S (R) \to R^S$ for a special class of rings. 

\begin{proposition}
\label{t11}
Let $A$ be a ring which for each prime $p \in S$ is equipped with a ring endomorphism $\phi_p$ lifting Frobenius i.e. with $\phi_p (a) \equiv a^p \mod p A$. Then we have
\[
\Gh_S (E_S (A)) = \{ (x_n) \in A^S \mid \phi_p (x_{n/p}) \equiv x_n \mod p^{\nu_p (n)} A \quad \text{for all} \; p \mid n \} \; .
\]
\end{proposition}

\begin{proof}
This is a restatement of Dwork's Lemma 1.1 from \cite{H}. The proof in \cite{H} applies to non-unital rings as well.
\end{proof}

\begin{examp}
For any ring $R$, the ring $ \Z R$ is equipped with commuting Frobenius lifts $\phi_p$ for all prime numbers $p$ defined by setting $\phi_p [r] = [r^p]$.
\end{examp}

We can now state a Cartier--Dieudonn\'e lemma as in \cite{I} (1.3.16) in our context. c.f. \cite{H} p. 13.

\begin{proposition}
\label{t12}
Let $A$ be a ring without $S$-torsion which is equipped with commuting Frobenius lifts $\phi_p$ for the primes $p \in S$. For $n \in S$ set $\phi_n = \prod_p \phi^{\nu_p (n)}_p$. Then the map $f_S : A \to A^S , f_S (a) = (\phi_n (a))_{n \in S}$ defines a ring homomorphism $f_S : A \to X_S (A) = E_S (A)$ which is functorial in $A$. For all $n \in S$ the diagram
\begin{equation}
\label{eq:28}
\xymatrix{
A \ar[r]^-{f_{S}} \ar[d]_{\phi_n} & E_{S} (A) \ar[d]^{F_n} \\
A \ar[r]^-{f_{S/n}} & E_{S/n} (A)
}
\end{equation}
commutes.
\end{proposition}

\begin{proof}
By Corollary \ref{t6} we have $X_S (A) = E_S (A)$ and hence the ghost map $\Gh_S : E_S (A) \to A^S$ is injective. The family $(x_n) = (\phi_n (a)) \in A^S$ satisfies the equations $\phi_p (x_{n/p}) = x_n$ for all $p \mid n$. Hence the assertion follows from Proposition \ref{t11}. 
\end{proof}

\begin{examp}
For an arbitrary ring $R$ the ring homomorphism
\[
f_S : \Z R \longrightarrow E_S (\Z R)
\]
of the preceding proposition is given by the formula
\begin{equation}
\label{eq:29}
f_S \Big( {\te \sum} n_r [r] \Big) = {\te \sum} n_r \langle [r] \rangle_S \; .
\end{equation}
Note the difference to the non-linear Teichm\"uller character for $\Z R$
\[
\langle \rangle_S \Big( {\te \sum} n_r [r] \Big) = \langle {\te \sum} n_r [r] \rangle_S
\]
which also splits the residue map $E_S (\Z R) \to \Z R$. The following diagram commutes
\begin{equation}
\label{eq:30}
\xymatrix{
\Z R \ar[r]^{f_S} \ar[dr]_{\alpha_S} & E_S (\Z R) \ar[d]^{E_S (\pi)} \\
 & E_S (R) \; .
 }
\end{equation}
Here $\alpha_S$ is the unique ring homomorphism extending the multiplicative Teichm\"uller map $\langle \rangle_S : R \to E_S (R)$ to $\Z R$. 
\end{examp}

\begin{cor}
\label{t13}
For any ring $A$ as in Proposition \ref{t12}, the map
\[
\Phi_S : A^S \longrightarrow E_S (A) \; , \; \Phi_S (\bfa) = \sum_{n \in S} V_n (f_{S/n} (a_n))
\]
is an isomorphism of abelian groups.
\end{cor}

\begin{proof}
The map $\Phi_S$ is $\Z$-linear and $\Phi_S (\bfa)_m = \sum_{n \mid m} n \phi_{m/n} (a_n)$. The same inductive argument as the one for $\Psi_S$ in the proof of Proposition \ref{t2}, b) shows that $\Phi_S$ is injective. \\
By Proposition \ref{t11} any element $\bfx \in X_S (A) \subset A^S$ satisfies the congruences $\phi_p (x_{m/p}) \equiv x_m \mod p^{\nu_p (m)} A$ for all $p \mid m$. We must show that for some $\bfa \in A^S$ we have
\begin{equation}
\label{eq:31} 
\sum_{n \mid m} n \phi_{m/n} (a_n) = x_m \quad \text{for all} \; m \in S \; .
\end{equation}
For $m = 1$ we take $a_1 = x_1$. Fix $i \ge 0$ and assume that elements $a_n \in A$ have been found for all $n \in S$ with $\nu (n) \le i$ such that \eqref{eq:31} holds for all $m \in S$ with $\nu (m) \le i$. Let $\nu (m) = i+1$. Then
\begin{align*}
x_m & \equiv \phi_p (x_{m/p}) \mod p^{\nu_p (m)} A \quad \text{for} \; p \mid m \\
& \equiv \sum_{n \mid \frac{m}{p}} n \phi_{m/n} (a_n) \mod p^{\nu_p (m)} A \quad \text{by the induction hypotheses.}
\end{align*}
Hence
\[
x_m \equiv \sum_{n \| m } n \phi_{m/n} (a_n) \mod p^{\nu_p (m)} A \quad \text{for} \; p \mid m
\]
and therefore
\begin{equation} \label{eq:34}
x_m \equiv \sum_{n \| m} n \phi_{m/n} (a_n) \mod m A \; .
\end{equation}
Hence there is an element $a_m \in A$ which upgrades the congruence \eqref{eq:34} to an equality \eqref{eq:31}.
\end{proof}

\begin{rem}
In particular the non-linear map $\Psi_S : (\Z R)^S \to (\Z R)^S$ of \eqref{eq:6a} and the $\Z$-linear map $\Phi_S : (\Z R)^S \to (\Z R)^S$ have the same image $X_S (\Z R)$. However neither of the images $\Psi_S (I^S)$ and $\Phi_S (I^S)$ is contained in the other in general, even if $R$ is an $\F_p$-algebra and $S = \{ 1 , p , p^2 , \ldots \}$, so that $\phi_p$ respects $I$. The $\Z$-linear composition
\[
\varphi_S : (\Z R)^S \xrightarrow{\overset{\Phi_S}{\sim}} X_S (\Z R) \twoheadrightarrow E_S (R)
\]
is given by $\varphi_S (\bfa) = \sum_{n \in S} V_n (\alpha_S (a_n))$ where $\alpha_S : \Z R \to E_S (R)$ is the unique extension of the multiplicative Teichm\"uller map $\langle \rangle_S : R \to E_S (R)$ to a ring homomorphism on $\Z R$. This leads to a new presentation of $E_S (R)$ because it is easy to define a multiplication on $(\Z R)^S$ so that $\varphi_S$ becomes a ring homomorphism. However the kernel $\Ker \varphi_S = \Phi^{-1}_S (X_S (I))$ does not seem to have a simple description in general. For example, if $R$ is an $\F_2$-algebra and $S = \{ 1,2,4 \}$ then
\[
\Ker \varphi_S = \{ (a_1 , a_2 , a_4) \in (\Z R)^S \mid a_1 \in I \; , \; \delta (a_1) + a_2 \in I \; , \; \delta (\delta (a_1) + a_2) + a_4 + a^2_2 \in I \} \; .
\]
Here $\delta (a) = 2^{-1} (\phi_2 (a) - a^2)$ for $a \in \Z R$.
\end{rem}

\section{The case of $\F_p$-algebras} \label{sec:5}
We have $F_n V_n = n$ but in general $V_n F_n \neq n$. However the following is true:

\begin{proposition}
\label{t9}
Let $R$ be an $\F_p$-algebra with Frobenius endomorphism $\phi_p$ and let $S$ be any truncation set with $p \in S$.\\
a) We have $V_p F_p = p$ on $E_S (R)$.\\
b) There is a commutative diagram
\[
\xymatrix{
E_S (R) \ar[rr]^{F_p} \ar[dr] & & E_{S/p} (R) \\
& E_{S/p} (R) \ar[ur]_{E_{S/p} (\phi_p)} & 
}
\] 
In particular if $S = S / p$ we have $F_p = E_S (\phi_p)$ on $E_S (R)$.
\end{proposition}

\begin{proof}
The projection $E (R) \twoheadrightarrow E_S (R)$ commuting with $V_p$, $F_p$ and the maps induced by $\phi_p$ it suffices to consider the case $S = \N$. \\
a) Recall the isomorphism $\nu_R : \hU (R) \silo E (R)$ from Theorem \ref{jc} for the presentation \eqref{eq:i1a}. We can check the assertion $V_p F_p = p$ in $\hat{U} (R) = 1 + tR [[t]]$. Because of the homeomorphism \eqref{eq:x11} and Proposition \ref{t3n} it suffices to do so on the elements of the form $1 - rt^n = V_n (1 - rt)$. For $p \mid n$ we have $V_p F_p V_n = p V_p V_{n/p} F_1 = pV_n$ by the relation in Proposition \ref{t1} b) which also holds in $E (R)$. For $p \nmid n$ using Propositions \ref{t1} a) and \ref{t3n} we have 
\[
V_p F_p V_n (1-rt) = V_p V_n F_p (1 - rt) = V_{pn} (1 - r^p t) = 1 - r^p t^{pn} \; .
\]
Since $R$ is an $\F_p$-algebra we have $1 - r^p t^{pn} = (1 - rt^n )^p$ in $\hat{U} (R)$ and hence $V_p F_p = p$ on $1 - rt^n$.\\
b) It suffices to check the relation $F_p = E (\phi_p)$ on the elements $1 - rt^n = V_n (1 - rt)$. Since $F_p$ commutes with $V_n$ on $E (R)$ by part a) and Proposition \ref{t1} and since $E (\phi_p)$ commutes with $V_n$ as well we may assume $n = 1$. By definition $F_p (1 - rt) = 1 - r^pt = E (\phi_p) (1 - rt)$ and we are done.
\end{proof}

The multiplicative Teichm\"uller map $\langle \rangle : R \to E_S (R)$ induces a homomorphism of rings $\alpha : \Z R \to E_S (R)$. Even if $S$ is finite, $\alpha$ will not be surjective in general. However we have the following theorem which is a consequence of the proposition by Umberto Zannier in the appendix of this paper.

\begin{theorem}
\label{t11n}
Let $R$ be an $\oF_p$-algebra and $S$ a finite truncation set. Then the map $\alpha : \Z R \to E_S (R)$ is surjective. For arbitrary truncation sets $S$ the image is dense.
\end{theorem}

\begin{proof}
The second assertion follows from the first one by using the topological isomorphism \eqref{eq:8}. For finite truncation  sets $T \subset S$ the projections $E_S (R) \to E_T (R)$ are surjective e.g. by Proposition \ref{tx2}. Hence we may assume that the finite truncation set $S$ has the form $S_m = \{ 1 , \ldots , m \}$ for some $m \ge 1$. According to Theorem \ref{jc} there is a natural isomorphism of groups where $\hU_m (R) = 1 + t^{m+1} R [[t]]$ and $\hU (R) = \hU_0 (R)$:
\[
\nu_{R,S_m} : \hat{U} (R) / \hat{U}_m (R) \silo E_{S_m} (R) \; .
\]
Via $\nu_{R,S_m}$ the Teichm\"uller map for $E_{S_m} (R)$ corresponds to the map \\
$R \to \hat{U} (R) / \hat{U}_m (R)$ which sends $\rho$ to the class of $(1 - \rho t)$. Under the induced additive homomorphism $\Z R \to \hat{U} (R) / \hat{U}_m (R)$ an element of the form $\sum_i [\rho_i]$ with $\rho_i \in R$ is therefore sent to $\prod_i (1 - \rho_i t)$. The proposition in the appendix by Umberto Zannier states that polynomials of this form represent all classes in $\hat{U} (R) / \hat{U}_m (R)$ and we are done.
\end{proof}

\begin{remark}
\label{t11nn} \rm 
For an $\F_p$-algebra $R$ and $p^r > m$ we have in $\hat{U} (R)$
\[
(1 - at)^{-n} \equiv (1 - a^{p^r} t^{p^r}) (1 -a t)^{-n} \equiv (1 - at)^{p^r-n} \mod \hat{U}_m (R) \; .
\]
It follows that the image of $\Z R$ in $\hat{U} (R) / \hat{U}_m(R) \cong E_{S_m} (R)$ is the same as the image of the subset $(\Z R)^+ = \{ \sum n_r [r] \mid n_r \ge 0 \}$. The counting argument in Remark 2 of the appendix therefore shows that the map $\Z R \to E_{S_m} (R)$ is not surjective for $R = \F_p$ if $m \ge 2p -1$. 
\end{remark}

\section{The $p$-typical case for perfect $\F_p$-algebras} \label{sec:6}
In this section we fix a prime number $p$ and consider the truncation sets $S_p = \{ 1 , p, p^2 , \ldots \}$ and $S^{(n)}_p = \{ 1 , p , \ldots , p^{n-1} \}$ for $n\ge 1$ . As in the introduction the indexing is now by the exponents of $p$. We write $E (R)$ for $E_{S_p} (R)$ and $E_n (R)$ for $E_{S^{(n)}_p} (R)$ etc. and set $V = V_p$ and $F = F_p$. Corollary \ref{t27n} implies that for $n, m \ge 1$ there are exact sequences of additive groups
\[
0 \longrightarrow E_m (R) \xrightarrow{V^n} E_{n+m} (R) \xrightarrow{\proj} E_n (R) \longrightarrow 0
\]
and
\[
0 \longrightarrow E (R) \xrightarrow{V^n} E (R) \xrightarrow{\proj} E_n (R) \longrightarrow 0 \; .
\]
In the rest of this section we assume that $R$ is a commutative possibly non-unital $\F_p$-algebra. We have $FV = p$ and by Proposition \ref{t9} also $VF = p$. Moreover $F = E (\phi)$ on $E (R)$ where $\phi : R \to R , \phi (r) = r^p$ is the Frobenius map. 

The multiplicative Teichm\"uller map $\langle \rangle : R \xrightarrow{[]} \Z R \xrightarrow{\langle \rangle} X (\Z R) \to E (R)$ induces a homomorphism of rings $\alpha : \Z R \to E (R)$ such that the following diagram commutes
\begin{equation}
\label{eq:35n}
\xymatrix{
\Z R \ar[r]^{\alpha} \ar[d]_{\pi} & E (R) \ar[d] \\
R \ar@{=}[r] & E_1 (R) \; .
}
\end{equation}

\begin{proposition} 
\label{tx15}
Let $I$ be the kernel of the natural projection $\pi : \Z R \to R$. The map $\alpha$ induces a homomorphism
\begin{equation}
\label{eq:37}
\alpha_n : \Z R / I^n \longrightarrow E_n (R) \; .
\end{equation}
\end{proposition}

\begin{proof}
We have $\alpha (I) \subset \Ker (E (R) \to R) = V E (R)$ by diagram \eqref{eq:35n} and hence
\[
\alpha (I^n) \subset \alpha (I)^n \subset (V E (R))^n \subset V^n E (R) \; .
\]
The last inclusion holds because the relations $p = FV = VF$ and Proposition \ref{t1}, c) imply that
\begin{equation}
\label{eq:36}
(V^{\nu} E (R)) (V^{\mu} E (R)) \subset V^{\nu + \mu} E (R) \quad \text{for all} \; \nu , \mu \ge 0 \; .
\end{equation}
\end{proof}

The ghost map $\Gh : E(R) \to R^{\N_0}$ can be improved as follows. We have inclusions
\[
X (I) \subset \prod^{\infty}_{\nu =0} I^{\nu +1} \subset I^{\N_0}
\]
because the $\nu$-th component ($p^{\nu}$-th in the previous section) of $V^{\mu} \langle a \rangle$ is zero for $\nu < \mu$ and equal to $p^{\mu} a^{p^{\nu -\mu}} \in I^{\nu +1}$ for $\nu , \mu \ge 0$ and $a \in I$. Note here that $p \in I$ since $R$ is an $\F_p$-algebra. Similarly $X_n (I) \subset \prod^{n-1}_{\nu =0} I^{\nu+1}$. 

\begin{definition}
\label{t17}
For an $\F_p$-algebra $R$ the refined ghost maps are the induced homomorphisms of rings
\[
\Gh_{\re} : E (R) \to \prod^{\infty}_{\nu =0}  \Z R / I^{\nu+1} \quad \text{and} \quad \Gh_{\re} : E_n (R) \to \prod^{n-1}_{\nu =0} \Z R / I^{\nu +1} \; .
\]
\end{definition}

By $\phi : \Z R \to \Z R$ we also denote the map induced by the Frobenius $\phi : R \to R$. 
\begin{lemma}
\label{t18}
The following diagram commutes
\[
\xymatrix{
 \Z R / I^n \ar[r]^{\alpha_n} \ar[d]_{\phi^{n-1}} & E_n (R) \ar[d]^{\Gh_{\re}} \\
 \Z R / I^n \; & \; \prod^{n-1}_{\nu =0} \Z R / I^{\nu+1} \ar[l]^-{\proj_{n-1}}
}
\]
\end{lemma}

\begin{proof}
For $r \in R$ we have $\alpha_n ([r] \mod I^n) = \langle r \rangle = \langle [r] \rangle \mod X_n (I)$. Hence
\[
(\Gh_{\re} \verk \alpha_n) ([r] \mod I^n) = ([r^{p^{\nu}}] \mod I^{\nu+1})_{0 \le \nu \le n-1}
\]
and the assertion follows. 
\end{proof}

Recall that an $\F_p$-algebra $R$ is called perfect if its Frobenius endomorphism $\phi$ is an isomorphism.

\begin{cor}
\label{t19}
For a perfect $\F_p$-algebra $R$ the map $\alpha_n : \Z R / I^n \silo E_n (R)$ is an isomorphism whose inverse is the composition $\phi^{1-n} \verk \proj_{n-1} \verk \Gh_{\re}$. The induced map $\halpha : \varprojlim_n \Z R / I^n \silo E (R)$ is a topological isomorphism of topological rings. 
\end{cor}

\begin{rem}
Since we know that $E_n (R) \cong W_n (R)$ we recover Corollary 7 of \cite{CD}. 
\end{rem}

\begin{proof}
Since $\phi$ is injective, Lemma \ref{t18} implies that $\alpha_n$ is injective. Every element of $E_n (R)$ has the form $\sum^{n-1}_{\nu =0} V^{\nu} \langle r_{\nu} \rangle$ for $r_{\nu} \in R$. For $r \in R$, using Proposition \ref{t9}, we obtain
\[
V \langle \phi (r) \rangle = V E_{n-1} (\phi) \langle r \rangle = VF \langle r \rangle = p \langle r \rangle \quad \text{in} \; E_n (R) \; .
\]
Thus $V \langle r \rangle = p \langle \phi^{-1} (r) \rangle$ and therefore
\begin{equation} \label{eq:55.5}
V^{\nu} \langle r \rangle = p^{\nu} \langle \phi^{-\nu} (r) \rangle = \alpha_n (p^{\nu} [\phi^{-\nu} (r)]) \; .
\end{equation}
Hence $V^{\nu} \langle r \rangle$ is in the image of $\alpha_n$ and therefore $\alpha_n$ is also surjective. By lemma \ref{t18} its inverse is $\phi^{1-n} \verk \proj_{n-1} \verk \Gh_{\re}$. The rest is clear.
\end{proof}

\begin{cor}
\label{t65}
For a perfect $\F_p$-algebra $R$, the inverse of the map $\alpha_n : \Z R / I^n \silo E_n (R) \cong W_n (R)$ is given by the formula 
\[
\alpha^{-1}_n (r_0 , \ldots , r_{n-1}) = \sum^{n-1}_{\nu = 0} p^{\nu} [\phi^{-\nu} (r_{\nu})] \mod I^n \; .
\]
\end{cor}

\begin{proof}
Using formula \eqref{eq:55.5} we have
\[
\alpha_n \Big( \sum^{n-1}_{\nu = 1} p^{\nu} [\phi^{-\nu} (r_{\nu})] \mod I^n\Big) = \sum^{n-1}_{\nu =0} V^{\nu} \langle r_{\nu} \rangle \equiv (r_0 , \ldots , r_{n-1}) \in W_n (R) \; .
\]
\end{proof}

Corollary \ref{t19} can be generalized to more general resolutions of $R$. Let $A$ be a ring without $p$-torsion which is equipped with a ring endomorphism $\phi$ lifting the Frobenius endomorphism of $A / pA$. Assume that we have an exact sequence
\begin{equation}
\label{eq:xxx1}
0 \longrightarrow I \longrightarrow A \xrightarrow{\pi} R \longrightarrow 0 \; .
\end{equation}
Consider the composition
\[
\alpha : A \xrightarrow{f} E (A) \xrightarrow{E (\pi)} E (R) \; .
\]
where $f = f_{S_p}$ is the ring homomorphism of Proposition \ref{t12}. Note that for the standard resolution by $A = \Z R$ this map $\alpha$ is the same as in \eqref{eq:35n} above. This follows from diagram \eqref{eq:30}. It follows from the definition of $f$ that $\alpha (I)\subset \Ker (E (R) \to R) = V E (R)$ and as in the proof of Proposition \ref{tx15} we obtain an induced ring homomorphism:
\[
\alpha_n : A / I^n \longrightarrow E_n (R) \; .
\]
The same arguments as before using the corresponding refined ghost map
\[
\Gh_{\re} : E_n (R) \longrightarrow \prod^{n-1}_{\nu =0} A / I^{\nu + 1}
\]
and the commutative diagram
\[
\xymatrix{
A / I^n \ar[rr]^{\alpha_n} \ar[d]_{\phi^{n-1}} & & E_n (R) \ar[d]^{\Gh_{\re}} \\
A / I^n & & \prod^{n-1}_{\nu = 0} A / I^{\nu+1} \ar[ll]_{\proj_{n-1}}
}
\]
imply the following generalization of Corollary \ref{t19}.

\begin{theorem}\label{t55}
Let $R$ be a perfect $\F_p$-algebra and consider a resolution \eqref{eq:xxx1} as above such that $\phi$ is an isomorphism of $A$. Then the map $\alpha_n : A / I^n \to E_n (R)$ is an isomorphism with inverse $\phi^{1-n} \verk \proj_{n-1} \verk \Gh_{\re}$ and we have
\begin{equation}
\label{eq:53a}
E_n (\phi) \verk \alpha_n = \alpha_n \verk \phi \; .
\end{equation} 
The induced map $\halpha : \varprojlim A / I^n \to E (R)$ is a topological isomorphism of topological rings.
\end{theorem}

\begin{proof}
The map $\phi$ being an isomorphism on $A / I^n$ it follows that $\alpha_n$ is injective. Using Corollary \ref{t13} we see that $E_n (R)$ is generated by the elements $V^{\nu} \alpha_{n - \nu} (a) = E_n (\pi) V^{\nu} (f (a))$ for $0 \le \nu < n$ and $a \in A$. The defining property \eqref{eq:28} of the map $f$ and Proposition \ref{t9}, a) show that
\[
V \alpha_{n-1} (\phi (a)) = E_n (\pi) Vf (\phi (a)) = E_n (\pi) VF (f (a)) = VF \alpha_n (a) = p \alpha_n (a) \; .
\]
Hence $V \alpha_{n-1} (a) = p \alpha_n (\phi^{-1} (a))$ and therefore $V^{\nu} \alpha_{n-\nu} (a) = p^{\nu} \alpha_n (\phi^{-\nu} (a))$. It follows that the map $\alpha_n$ is surjective. Formula \eqref{eq:53a} follows similarly using Proposition \ref{t9} b).
\end{proof}

\begin{rem}
The proof of Theorem \ref{t55} specialized to $A = \Z R$ differs from the one of Corollary \ref{t19} regarding the surjectivity of $\alpha_n$. The next Corollary corresponds to Proposition 1.3.22 of \cite{I} or Proposition 1.13 of \cite{H}.
\end{rem}

\begin{cor} \label{t65nn}
Let $A$ be a commutative ring without $p$-torsion and equipped with an endomorphism $\phi$ lifting the Frobenius endomorphism of $R = A / pA$. Assume that $R$ is a perfect $\F_p$-algebra and let $\hA = \varprojlim_n A / p^n A$ be the $p$-adic completion of $A$. Then the maps
\[
\alpha_n : A / p^n A \longrightarrow E_n (A / pA) \quad \text{and} \quad \halpha : \hA \longrightarrow E (A / pA)
\]
are isomorphisms. 
\end{cor}

\begin{proof}
The completion $\hA$ has no $p$-torsion as well and $\hA / p^n \hA = A / p^n A$. An easy inductive argument shows that the Frobenius lift is an isomorphism on each $A / p^n A$ and hence on $\hA$. Applying the preceding Theorem to the presentation $0 \to p \hA \to \hA \to R \to 0$ the Corollary follows.
\end{proof}

Thus Theorem \ref{t55} is a joint generalization of the well known Corollary \ref{t65nn} in Witt vector theory and of Corollary \ref{t19} which was pointed out in \cite{CD} .

The intersection $\bigcap_{n \ge 0} I^n$ is an interesting but somewhat mysterious ideal in $\Z R$. Using the preceding results we have the following side result:

\begin{proposition}
\label{t56}
Consider the standard resolution $0 \to I \to \Z R \to R \to 0$ of a perfect $\F_p$-algebra $R$. Then the map $R \to E (\F_p R) , r \mapsto \langle [r] \rangle$ induces an injective homomorphism of rings $\Z R \hookrightarrow E (\F_p R)$. Viewing it as an inclusion we have
\[
\bigcap_{n \ge 0} I^n = \Z R \cap E (I / p \Z R) \quad \text{in} \; E (\F_p R) \; .
\]
\end{proposition}

\begin{proof}
Since $\Z R$ injects into $\widehat{\Z R}$, Corollary \ref{t65nn} implies that the natural map $\Z R \to E (\F_p R)$ is injective. We have a commutative diagram with exact rows where the vertical arrows are isomorphisms
\begin{equation}
\label{eq:xxx3}
\xymatrix{
0 \ar[r] & I^n / p^n \Z R \ar[d]^{\wr} \ar[r] & \Z R / p^n \Z R \ar[d]^{\wr} \ar[r] & \Z R / I^n \ar[d]^{\wr} \ar[r] & 0 \\
0 \ar[r] & E_n (I / p \Z R) \ar[r] & E_n (\F_p R) \ar[r] & E_n (R) \ar[r] & 0 
}
\end{equation}
This follows from Corollaries \ref{t28nn} and \ref{t19} and \ref{t65nn}. We view \eqref{eq:xxx3} as a diagram of projective systems over $n$ under the canonical projections. The system $(E_n (I / p \Z R)_{n \ge 1}$ has surjective transition maps and is therefore Mittag--Leffler. The same is then true for the isomorphic system $(I^n / p^n \Z R)_{n \ge 1}$. With a direct proof this was observed at the end of \cite{CD} at least in the unital case. Passing to projective limits we get a diagram with exact rows:
\[
\xymatrix{
0 \ar[r] & \varprojlim_n I^n / p^n \Z R \ar[d]^{\wr} \ar[r] & \widehat{\Z R} \ar[d]^{\wr} \ar[r] & \varprojlim_n \Z R / I^n \ar[d]^{\wr} \ar[r] & 0 \\
0 \ar[r] & E (I / p \Z R) \ar[r] & E (\F_p R) \ar[r] & E (R) \ar[r] & 0 \\
 & & \Z R \ar@{_{(}->}[u] \ar[ur]^{\alpha} & & 
 }
\]
We conclude
\[
\bigcap_{n \ge 0} I^n = \Ker \alpha  = \Z R  \cap \Ker (E (\F_p R) \to E (R)) = \Z R \cap E (I / p \Z R) \; .
\]
\end{proof}

\section{The $p$-typical case, non-perfect $\F_p$-algebras} \label{sec:7}
In this section we study the canonical map $\alpha_n : \Z R / I^n \to E_n (R)$ in Proposition \ref{tx15} for certain classes of not necessarily perfect $\F_p$-algebras $R$. We first look at the kernel of $\alpha_n$. Consider the following ideals for $n \ge 1$
\[
I_n = \phi^{1-n} (I^n) = \{ a \in \Z R \mid \phi^{n-1} (a) \in I^n \} \supset I^n \; .
\]

\begin{theorem}
\label{t20}
Let $R$ be an $\F_p$-algebra whose Frobenius endomorphism $\phi$ is injective e.g. a reduced algebra. Then the following assertions hold:\\
a) $\Ker \alpha_n = I_n / I^n$ i.e. $\alpha : \Z R \to E (R)$ induces an injective homomorphism of rings $\beta_n : \Z R / I_n \hookrightarrow E_n (R)$
for any $n \ge 1$.\\
b) $I_1 = I$ and $I_n \cdot I_m \subset I_{n+m}$ for $n,m \ge 1$.\\
c) $I_n = \{ a \in \Z R \mid \phi^i (a) \in I^n \; \text{for some} \; i \ge 0 \}$.\\
d) For $n \ge 1$ we have $I_n \supset I_{n+1}$.
\end{theorem}

\begin{proof}
a) Lemma \ref{t18} implies that $\Ker \alpha_n \subset I_n / I^n$. For $a \in I_n$ we have $\phi^{n-1} (a) \in I^n$ and hence $E_n (\phi)^{n-1} (\alpha_n (\oa)) = \alpha_n (\phi^{n-1} (\oa)) = 0$ where $\oa = a \mod I^n$. The map $E_n (\phi)$ is injective on $E_n (R)$ because the bijection $\psi_n : R^n \silo E_n (R)$ given by $\psi_n (\bfr) = \sum^{n-1}_{i=0} V^i \langle r_i \rangle$ satisfies $\psi_n \verk (\phi \times \ldots \times \phi) = E_n (\phi) \verk \psi_n$ and $\phi \times \ldots \times \phi$ ($n$-times) is injective on $R^n$. Hence $\alpha_n (\oa) = 0$ and therefore $I_n / I^n \subset \Ker \alpha_n$.\\
b) follows from the definitions.\\
c) We need the "arithmetic derivation" $\delta : \Z R \to \Z R$ defined by 
\[
\delta (a) = \frac{1}{p} (\phi (a) - a^p)\; .
\]
It satisfies the following relations for $a,b \in \Z R$
\begin{equation}
\label{eq:38}
\delta (a+b) = \delta (a) + \delta (b) - \sum^{p-1}_{\nu =1} \frac{1}{p} {p \choose \nu} a^{\nu} b^{p-\nu}
\end{equation}
and
\begin{equation}
\label{eq:39}
\delta (ab) = \delta (a) \phi (b) + a^p \delta (b) \; .
\end{equation}
Formula \eqref{eq:38} implies that
\begin{equation}
\label{eq:40}
\delta (a+b) \equiv \delta (a) + \delta (b) \mod I^n \quad \text{for $a$ or $b$ in $I^n$.}
\end{equation}
Together with \eqref{eq:39} we find as in \cite{CD} that:
\begin{equation}
\label{eq:41}
\delta (I^n) \subset I^{n-1} \quad \text{for} \; n \ge 1 \; .
\end{equation}
We can now prove assertion c) by induction on $n \ge 1$. For all $i \ge 0$ we have $\phi^{-i} (I) = I$. Namely $\phi^i (a) \in I$ implies that $\phi^i (\oa) = 0$ in $R = \Z R / I$ for $\oa = a \mod I$. Since $\phi$ is assumed to be injective, we get $\oa = 0$ i.e. $a \in I$. Hence $\phi^{-i} (I) \subset I$. The reverse inclusion is obvious. Hence c) holds for $n = 1$. Assume that c) holds for some $n \ge 1$. If $a \in \Z R$ satisfies $\phi^i (a) \in I^{n+1}$ for some $i \ge 0$ then using \eqref{eq:41} we find
\[
\phi^i (\delta (a)) = \delta (\phi^i (a)) \in \delta (I^{n+1}) \subset I^n \; .
\]
Using the induction hypotheses we get $\delta (a) \in I_n$ i.e. $\phi^{n-1} (\delta (a)) \in I^n$. This implies
\[
\phi^n (a) - \phi^{n-1} (a)^p \in p I^n \subset I^{n+1} \; .
\]
We have $\phi^i (a) \in I^{n+1} \subset I^n$. Applying the induction hypotheses again we find that $a \in I_n$ i.e. $\phi^{n-1} (a) \in I^n$. Hence $\phi^{n-1} (a)^p \in I^{pn} \subset I^{n+1}$ and we conclude that $\phi^n (a) \in I^{n+1}$ i.e. that $a \in I_{n+1}$. The other inclusion in c) is clear.\\
Assertion d) follows from a) or c).
\end{proof}

The next result concerning the image of $\alpha_n$ is a special case of Theorem \ref{t11n} from the previous section which is directly implied by Umberto Zannier's proposition in the appendix.

\begin{theorem}
\label{t21}
Let $R$ be an $\oF_p$-algebra. Then the natural map $\Z R \to E_n (R)$ is surjective. If in addition the Frobenius endomorphism of $R$ is injective, the map $\beta_n : \Z R / I_n \to E_n (R)$ is an isomorphism. In this case $\alpha$ induces a topological isomorphism
\begin{equation}
\label{eq:42}
\halpha : \varprojlim_{n} \Z R / I_n \silo E (R) \; .
\end{equation}
\end{theorem}

Since $E_n (R) \cong W_n (R)$ and $E (R) \cong W (R)$ this gives new descriptions of the rings of $p$-typical Witt vectors for such algebras. Under the isomorphism $\halpha$ the Frobenius endomorphism $F = E (\phi)$ on $E (R)$ corresponds to the endomorphism of $\varprojlim_n \Z R / I_n$ induced by the Frobenius lift $\phi$ of $\Z R$. 

The endomorphism $V$ of $\varprojlim_n \Z R / I_n$ corresponding to the Verschiebung $V$ on $E (R)$ is obtained as follows. Let $\oR = \varinjlim_{\phi} R$ be the inductive limit of the inductive system $R \overset{\phi}{\hookrightarrow} R \overset{\phi}{\hookrightarrow} R \hookrightarrow \ldots$ It is a perfect ring. Consider the exact sequence $0 \to \oI \to \Z \oR \to \oR \to 0$. By Corollary \ref{t19} we have an isomorphism 
\[
\varprojlim_n \Z \oR / \oI^n \silo E (\oR) \; .
\]
Therefore the Verschiebung on $\varprojlim \Z \oR / \oI^n$ is given by $p \ophi^{-1}$ where $\ophi$ is the Frobenius lift to $\Z \oR$ of the Frobenius automorphism $\ophi$ of $\oR$. 

Viewing $R$ as embedded in $\oR$ in degree zero, we have an inclusion $E (R) = W (R) \subset W (\oR) = E (\oR)$ and a commutative diagram 
\begin{equation}
\label{eq:43}
\xymatrix{
\varprojlim_n \Z R / I_n \ar[r]^{\sim} \ar@{^{(}->}[d] & E (R) \ar@{^{(}->}[d]\\
\varprojlim_n \Z \oR / \oI^n \ar[r]^{\sim} & E (\oR) \; .
}
\end{equation}
In particular $V = p \ophi^{-1}$ respects $\varprojlim_n \Z R / I_n$ and induces the Verschiebung. Similarly, there is a commutative diagram
\[
\xymatrix{
\Z R / I_n \ar[r]^{\sim} \ar@{^{(}->}[d] & E_n (R) \ar@{^{(}->}[d] \\
\Z \oR / \oI^n \ar[r]^{\sim} & E_n (\oR) \; .
}
\]
It implies that we have $I_n = \Z R \cap \oI^n$, a fact which can also be proved directly with a little effort.
\section*{Appendix: A factorization result for polynomials}
\begin{center} {Umberto Zannier}
\end{center}

The aim of this Appendix is to prove the Proposition below, providing a certain product decomposition for representatives $\mod t^m$ of polynomials in $t$ over suitable rings. This result is used in Theorems \ref{t11n} and \ref{t21} of the present paper. We begin with notations and some lemmas.

Let $p$ be a prime number and let $\eo$ be the valuation ring of the maximal unramified extension of $\Q_p$ in a given algebraic closure $\oQ_p$. Then 
the residue field of $\eo$, denoted $k$,  is an algebraic closure of $\F_p$. 

\medskip

\noindent{\bf Lemma 1.} {\it Let $N,m$ be   positive integers. There exist a finite set $I$ and  elements $a_i,b_i\in \eo$, $i\in I$,  such that, for a variable $x$, we have in the ring $\eo[x]$, 
$$
\sum_{i\in I} (a_i+b_ix)^l\equiv 0 \mod {p^N},\qquad  l=1,\ldots ,  m-1,
$$
$$
\sum_{i\in I}(a_i+b_ix)^m\equiv mx\mod{p^N}.
$$
}

\medskip

\noindent{\it Proof.}  Let us first assume $m\not\equiv 1\mod p$ and let us express (the still unknown) $I$ as a disjoint union of finite sets $R,S,T$ with the properties that 

(i) $|S\cup T|$ is divisible by $p^N$, 

(ii) $a_i=0$ for $i\in T$, 

(iii) $b_i=0$ for $i\in R$, 

(iv) $b_i=1$ for $i\in S\cup T$.

Then $b_i^2=b_i$ for all $i\in I$, so that, as is readily checked,  the above requirements amount to find disjoint $R,S$ and $a_i\in \eo$, for $i\in U:=R\cup S$, such that
$$
\sum_Ua_i^l\equiv 0\mod{p^N}, \qquad l=1,\ldots ,m,
$$
$$
\sum_Sa_i^l\equiv 0\mod{p^N}, \qquad l=1,\ldots ,m-2,
$$
$$
\sum_Sa_i^{m-1}\equiv 1\mod{p^N}.
$$
$T$ is then chosen so that $|S \cup T|$ is divisible by $p^N$.

Now, we contend that we can certainly satisfy the second set of equations and the last equation. Indeed, since $p$ does not divide $m-1$, the group $\Theta$ of $(m-1)$-th roots of unity (in an algebraic closure of $\Q_p$)   is contained in $\eo$. Let us then define $S$ as a disjoint union of $g$ copies  $\Theta_1,\ldots ,\Theta_g$ of $\Theta$, putting $a_\theta=\theta$ for $\theta\in\Theta_h$. Then plainly the second set of equations is satisfied (actually with equality in place of congruence), whereas $\sum_Sa_i^{m-1}=(m-1)g$. Since $m-1$ is prime to $p$ we can choose the positive integer $g$ as an inverse to $m-1$ modulo $p^N$, proving the claim. 


 Finally, to satisfy also the first set of equations, we can choose $R=S'$ as a disjoint copy of $S$ setting $a_{i'}=\lambda a_i$, where $\lambda\in\eo$ is such that $\lambda^{m-1}=-1$. 

\medskip

Let us now consider the case when $m\equiv 1\mod p$. 
Now,  the set $\Gamma$ of $m$-th roots of unity is contained in $\eo$. We take $I$ to be  the disjoint union of $g$ copies  $\Gamma_1,\ldots ,\Gamma_g$  of $\Gamma$ for a number $g$ to be chosen later, and we put, for $\gamma\in\Gamma_h$, $a_\gamma =\gamma\alpha_h$, $b_\gamma:=\gamma \beta_h$, with $\alpha_h,\beta_h$ still to be chosen. 
Note that
$$
\sum_I(a_i+b_ix)^l=\sum_{h=1}^g(\alpha_h+\beta_hx)^l\sum_{\gamma\in\Gamma}\gamma^l.
$$
This vanishes for $l=1,\ldots ,m-1$ and is $=m\sum_{h=1}^g(\alpha_h+\beta_hx)^m$ for $l=m$. Thus it suffices that
$$
\sum_{h=1}^g(\alpha_h+\beta_hx)^m\equiv x\mod{p^N}.
$$
Put now $\beta_h=\alpha_h\xi_h$, so we want that $\sum_{h=1}^g\alpha_h^m(1+\xi_hx)^m \equiv x \mod p^N$, which may be rephrased by saying that the vector $(\alpha_h^m)_h\in \eo^g$ is orthogonal modulo $p^N$ to the vectors $(\xi_h^l)_h$, for $l=0,2,\ldots ,m$, whereas its scalar product with $(\xi_h)_h$ is $\equiv m^{-1}\mod{p^N}$. 

These last conditions are easy to satisfy, even taking $g=m+1$ (while we could take it much larger) and even replacing the congruences by equalities. Indeed,   we may choose the $\xi_h$ as  any $g=m+1$ elements of $\eo$ which are pairwise incongruent modulo $p\eo$. Then the matrix $(\xi_h^l)_{h,l=1}^g$ is a Vandermonde matrix with determinant invertible in $\eo$, and we may then invert and choose the $\alpha_h^m$ in $\eo$ as wanted. To actually have $\alpha_h\in\eo$ (not merely $\alpha_h^m$), we may even choose the $\xi_h$ so that none of the $\alpha_h^m$ so obtained is in $p\eo$: indeed, the solution of the said system of linear equations may be expressed as a vector whose coordinates are ratios  of determinants in the variables $\xi_h$, none of which vanishes identically modulo $p$; then we may choose the $\xi_h\in\eo$ so that none of the reductions of these determinants vanishes in the residue field $k$. With such a choice, and 
since $m$ is prime to $p$,  we may then extract arbitrary $m$-th roots of the $\alpha_h^m$ so obtained,  and finally obtain suitable $\alpha_h\in\eo$. 

This concludes the proof of the lemma.\beweisende

\medskip

\noindent{\bf Remark 1.} The very last part of the proof is actually not needed for the application that we have in mind. Omitting it would lead to a similar statement, with the maximal unramifield extension replaced by the algebraic closure, which would be  enough for the arguments of the Proposition below.

\bigskip

\noindent{\bf Lemma 2.} {\it Let $\Lambda:=\eo[x,t]$ with independent variables $x,t$, and let $m,n$ be positive integers. Then the element $1-xt^m\in\Lambda$ may be written, modulo the ideal $(p^n,t^{m+1})\Lambda$,  as a product of  (finitely many) terms of the form $1-ft$, $f\in\eo [x]$.}

\medskip

\noindent{\it Proof.}  Set  $N=n+B$ in Lemma 1, where $B$ is the maximum integer such that $p^B\le m$, and let us choose $a_i,b_i$ as in the conclusion therein. Put $f_i=a_i+b_ix\in\eo[x]$ and define $F(t):=(1-xt^m)\prod_{i\in I}(1-f_it)^{-1}$. We interpret, as we clearly may,  the rational function $F(t)$ as an element of $\eo[x][[t]]$.  Taking the logarithmic derivative  with respect to $t$ and multiplying by $-t$ we obtain
$$
-{tF'(t)\over F(t)}={mxt^m\over 1-xt^m}-\sum_{i\in I}{f_it\over 1-f_it}.
$$
In the ring $\eo[x][[t]]$, for $s\in\eo[x]$ we may expand ${st^e\over 1-st^e}$ as $\sum_{l=1}^\infty s^lt^{le}$; if we do this in the right hand side of the last displayed formula we obtain
$$
-{tF'(t)\over F(t)}=\sum_{r=1}^\infty mx^rt^{rm} -\sum_{l=1}^\infty (\sum_{i\in I}(a_i+b_ix)^l)t^l.
$$
Taking into account the properties claimed by Lemma 1, we then obtain
$$
{tF'(t)\over F(t)}\equiv 0\mod{(p^N,t^{m+1})},
$$
where now the congruence has to be read in the ring $\eo[x][[t]]$.  Since $F(t)$ lies in this ring, this shows that 
$tF'(t)\in (p^N,t^{m+1})\eo[x][[t]]$. This means that in the $t$-expansion of $tF'(t)$ all the terms $\beta t^e$ which appear  ($\beta\in \eo[x]$)  have either $e\ge m+1$ or $e\le m$ and $\beta\in p^N\eo[x]$. 

But these terms $\beta t^e$ are in turn precisely those of the form $e\alpha t^e$ where $\alpha t^e$ is a term in the $t$-expansion of $F(t)$. Then, if $e\le m$ we must have $e\alpha\in p^N\eo[x]$ and, by the definition of $N$, this implies either $e=0$ or $\alpha\in p^n\eo[x]$. Hence $F(t)\equiv F(0)=1 \mod{(p^n,t^{m+1})}$. This concludes the proof. \beweisende

\noindent{\bf Proposition.} {\it Let $k$ be as above an algebraic closure of $\F_p$, and let ${R}$ be a ring containing $k$. Let $Q\in 1+ t R [t]$ have degree $\le m$. Then there exists a finite product $Q_1:=\prod_i (1-\rho_i t)$, $\rho_i\in {R}$,  such that $Q\equiv Q_1\mod{t^{m+1}{R}[t]}$.}

\medskip

Here, if $R$ has no unit element we embed it into a unital ring $\tR$ and let $1$ be the unit element of $\tR$.

\noindent{\it Proof.}  We argue by induction on $m$, where for $m=1$ we may choose $Q_1=Q$. Let us suppose $m>1$ and the result true up to $m-1$.

Writing $Q=q+at^m$, with $\deg q\le m-1$ and $a\in {R}$, we may apply the conclusion to $q$, obtaining a product $q_1$ of the required shape such that $q-q_1$ is divisible by $t^m$ in ${R}[t]$.  Since each term $(1-\rho t)^{-1}$, $\rho\in {R}$, lies in $1 + tR[[t]]$, we may thus write
$$
Qq_1^{-1}=1-bt^m+O(t^{m+1}), 
$$
for a certain $b\in {R}$, meaning that the remainder term is in $t^{m+1}{R}[[t]]$. 

Let us now use Lemma 2, with $n=1$,  which allows us to  write an identity
$$
1-xt^m=\prod_{i\in I}(1-f_i(x)t) +p\lambda_1(x,t)+t^{m+1}\lambda_2(x,t),
$$
where $I$ is a finite set, $f_i\in \eo[x]$, $\lambda_1,\lambda_2\in \eo[x,t]$. We may reduce modulo $p$ this identity and, denoting reduction  with  a tilde,  we obtain
$$
1-xt^m=\prod_{i\in I}(1-\tilde f_i(x)t) +t^{m+1}\tilde \lambda_2(x,t),
$$
where now $\tilde f_i\in k[x]$, $\tilde \lambda_2\in k[x,t]$. Further, we may specialise $x=b$ and obtain
$$
1-bt^m=\prod_{i\in I}(1-\sigma_it)+ t^{m+1}q_2(t)
$$
where $\sigma_i:=\tilde f_i(b)\in {R}$ and $q_2(t):=\tilde \lambda_2(b,t)\in {R}[t]$.  

Clearly, using this last identity in the right hand side of the displayed formula for $Qq_1^{-1}$ and multiplying by $q_1$,  we obtain what is needed, completing the induction and the proof. \beweisende

\bigskip

\noindent{\bf Remark 2.} The Proposition would not generally hold for ${R}=\F_p$, at any rate for any $m\ge 2p-1$. Indeed, let $r$ be the integer such that $p^{r-1}\le m<p^r$, and consider the possible polynomials $Q_1$. They have the shape $\prod_{a\in\F_p^*}(1-at)^{b_a}$ for integers $b_a\ge 0$. However, since $(1-at)^{p^r}=1-at^{p^r}=1+O(t^{m+1})$, for our purposes we may suppose that $0\le b_a\le p^r-1$. Then the number of polynomials $Q_1$ so obtained is $p^{r(p-1)}$. On the other hand, the number of polynomials $Q\in \F_p[t]$ of degree $\le m$ and with $Q(0)=1$ is $p^m$. Hence if  the conclusion of the Proposition holds, we must have $m\le r(p-1)$ and in particular $p^{r-1}\le r(p-1)$. This last yields $r\le 2$, hence $m\le 2(p-1)$.

A similar argument works for any finite field. This also shows that the ring $R = \Z_p$ would not be suitable; however one can show that the field $R = \Q_p$ would be.

One can get positive results for small $m$, and seemingly a slightly more involved argument than the above proves that $m=p$ is the largest possible value for $\F_p$ and $\Z_p$.

\medskip

 \medskip

A simple argument also shows that the Proposition does not hold for ${R}=\R$, already for $m=2$. Indeed, if $1+t^2$ is a product of terms $1-\rho_i t$, $\rho_i\in\R$ for $i=1,\ldots ,n$, modulo $t^3$, then $\sum_i\rho_i=0$, $\sum_{i<j}\rho_i\rho_j=1$. But then $\sum_i\rho_i^2=(\sum_i\rho_i)^2-2(\sum_{i<j}\rho_i\rho_j)=-2$, impossible. 

However, if one allows negative exponents of the factors $(1-\rho t)$ then $\R$ would become suitable. In positive characteristic, instead, this would not change anything, for $(1-\rho t)^{-1}\equiv (1-\rho t)^{q-1}\mod {t^q}$, for any $q$ which is a power of the characteristic. 


\end{document}